\documentclass[11pt]{amsart}

\usepackage[usenames,dvipsnames,svgnames,table]{xcolor}
\usepackage[hyphens]{url}
\usepackage[pagebackref,linktocpage=true,colorlinks=true,linkcolor=Blue,citecolor=BrickRed,urlcolor=RoyalBlue]{hyperref}
\usepackage[msc-links,abbrev]{amsrefs}
\usepackage{amsmath,amsthm,amssymb}
\usepackage{mathtools,mathrsfs}
\usepackage{booktabs,placeins}

\newtheorem{theorem}{Theorem}[section]
\newtheorem{lemma}[theorem]{Lemma}
\newtheorem{proposition}[theorem]{Proposition}

\theoremstyle{definition}
\newtheorem{note}[theorem]{Note}

\DeclareMathOperator{\tr}{tr\vphantom{g}}
\DeclareMathOperator{\diver}{div}

\DeclareMathOperator{\sgn}{sgn}
\DeclareMathOperator{\dist}{dist}

\newcommand{\R}{\mathbf R}
\newcommand{\Sph}{\mathbf S}
\newcommand{\B}{\mathbf B}
\newcommand{\cE}{\mathcal E}
\newcommand{\bcE}{\overline{\cE}}
\newcommand{\eps}{\varepsilon}
\newcommand{\1}{\mathbf 1}
\newcommand{\C}{\mathcal{C}}

\newcommand{\M}{\mathcal M}
\newcommand{\CC}{\mathcal C}

\renewcommand{\tilde}{\widetilde}

\title[The isoperimetric inequality and CMC hypersurfaces]
{The isoperimetric inequality and CMC hypersurfaces\\in Cartan--Hadamard manifolds}

\author{Shibing Chen}
\address{School of Mathematical Sciences, University of Science and Technology of China, Hefei, Anhui 230026, China}
\email{chenshib@ustc.edu.cn}

\author{Mohammad Ghomi}
\address{School of Mathematics, Georgia Institute of Technology,
Atlanta, GA 30332}
\email{ghomi@math.gatech.edu}
\urladdr{ghomi.math.gatech.edu}

\author{Peng Wang}
\address{School of Mathematics and Statistics, FJKLAMA, Key Laboratory of Analytical Mathematics and Applications, Fujian Normal University, Fuzhou, China}
\email{pengwang@fjnu.edu.cn,\;netwangpeng@163.com}

\date{\today\,(Last Typeset)}


\begin{document}

\begin{abstract}
We show that the sharp Euclidean isoperimetric inequality holds for domains in complete
simply connected Riemannian $n$-manifolds of nonpositive sectional
curvature, $3\leq n\leq9$, which establishes the Cartan--Hadamard conjecture in these dimensions. The main
step is a sharp inequality for constant-mean-curvature hypersurfaces,
proved via integrals over pairs of boundary points,  together with estimates for Jacobi fields along geodesic
chords of the boundary. The weights in these integrals depend on the
length of the chord and  its angles with the boundary, and are chosen dimension by dimension, with 
a Green function pole. 
The inequality
persists for boundaries of isoperimetric regions trapped in geodesic balls, whose
mean curvature is constant only on the free part and which may have
singularities. The isoperimetric-profile argument  of Kleiner and Ghomi--Spruck completes the proof. 
\end{abstract}

\maketitle

\section{Introduction}

A Cartan--Hadamard manifold $M^n$ is a complete simply connected
Riemannian $n$-space of nonpositive sectional curvature. The
Cartan--Hadamard conjecture \cite{aubin1976,burago-zalgaller1988,gromov1999}
asserts that regions in $M^n$ satisfy the Euclidean
isoperimetric inequality. We prove this in the following dimensions. Let $\B^n$
be the unit ball in Euclidean space $\R^n$, and $\Sph^{n-1}:=\partial \B^n$ be the unit sphere.

\begin{theorem}\label{thm:main}
Let $\Omega\subset M^n$, $3\leq n\leq9$, be a bounded set, and
$\Gamma\coloneqq\partial\Omega$. Suppose that the perimeter $|\Gamma|$
is finite and the volume $|\Omega|>0$. Then
\begin{equation}\label{eq:main}
\frac{|\Gamma|^n}{|\Omega|^{n-1}}
\geq
\frac{|\Sph^{n-1}|^n}{|\B^n|^{n-1}},
\end{equation}
with equality only if $\Omega$ is isometric to a Euclidean ball.
\end{theorem}

The Cartan--Hadamard conjecture had been proved earlier, via different methods, only in dimensions $n=2$, $3$, $4$, and $5$
by Weil \cite{weil1926}, Kleiner \cite{kleiner1992}, Croke
\cite{croke1984}, and the authors \cite{chen-ghomi-wang} respectively.  Theorem~\ref{thm:main} gives a unified proof of these results for $n\geq 3$ by refining the approach in  \cite{chen-ghomi-wang}.  The main step  is again a sharp
inequality for constant-mean-curvature (CMC) hypersurfaces:

\begin{theorem}\label{thm:CMC}
Let $\Gamma^{n-1}\subset M^n$, $3\leq n\leq9$, be a smooth compact
embedded hypersurface whose mean curvature $H$ is a positive constant.
Then
\begin{equation}\label{eq:CMC}
|\Gamma|\geq\left(\frac{n-1}{H}\right)^{n-1}|\Sph^{n-1}|.
\end{equation}
Equality holds only if $\Gamma$ bounds a Euclidean ball of radius
$(n-1)/H$.
\end{theorem}

The above result may be regarded as the analogue of Alexandrov's celebrated sphere theorem in $\R^n$ \cite{alexandrov1962}. As in \cite{chen-ghomi-wang}, the proof is based on combining two integral identities over $\Gamma\times\Gamma$: a Minkowski--Green identity and a degree identity. Here the Minkowski--Green identity uses weights which depend not only on the lengths of chords of $\Gamma$ but also on their angles with $\Gamma$, and the degree identity employs more general weight functions as well. Another major difference with \cite{chen-ghomi-wang} is that isoperimetric regions may be singular in dimensions $n\geq 8$; however, we show that our identities persist. 

In
Section~\ref{sec:jacobi-section} we begin the study of the chords of $\Gamma$ and their Jacobi fields.
Sections~\ref{sec:comparisons} and~\ref{subsec:joint} prove
two comparison estimates for these fields, which extend the estimate in  \cite{chen-ghomi-wang}. First we establish an
endpoint comparison which bounds the Jacobian of the exponential
map along a chord by the Hessians of the distance function at its
endpoints. Then we sharpen that result by taking into account
the effect of curvature on the angles between a chord and $\Gamma$. Section~\ref{sec:two-point} establishes the
two integral identities, and Section~\ref{sec:singular-section}
extends them to the boundaries of isoperimetric regions trapped inside geodesic balls.

In Section~\ref{sec:CMC} we combine the two identities through a
calibration scheme and establish our reduction result
(Theorem~\ref{thm:general-reduction}): a calibration in dimension
$n$, that is, a choice of weights for which a certain pointwise
inequality holds along every chord, yields both \eqref{eq:main} and \eqref{eq:CMC} in that dimension. The CMC inequality \eqref{eq:CMC}
follows quickly by integration, and its
equality case gives the rigidity; the isoperimetric inequality \eqref{eq:main}
then follows through the profile method of Kleiner
\cite{kleiner1992} and Ghomi--Spruck \cite{ghomi-spruck2022},
which reduces it to the CMC inequality for the isoperimetric regions. Section~\ref{sec:beyond-seven} exhibits
the calibrations for $3\leq n\leq9$, using the endpoint estimates, and completes the proofs.

The calibrations are quite simple in dimensions $3$ and $5$ but require longer computations in other dimensions, which are included in Appendix~\ref{app:calibrations} and certified in  the accompanying Mathematica notebook \cite{chen-ghomi-wang-calibrations}.  We have verified that calibrations exist at least up to dimension $13$, but do not include those computations here due to their length. 
In a sequel  \cite{chen-ghomi-wang-generalized}, we will show that our methods can be refined to obtain the generalized form of the Cartan--Hadamard conjecture in dimensions $n=4$ and $5$: when the curvature of $M$ is bounded above by a negative constant, inequality \eqref{eq:main} holds in terms of balls in the corresponding hyperbolic space. This has been known only in dimensions 2 \cite{beckenbach-rado1933} and 3 \cite{kleiner1992}. See \cite{ghomi-stavroulakis2026,ghomi2026-nullity} for recent partial results which hold in all dimensions and \cite{ghomi-spruck2022,kloeckner-kuperberg2019,ritore2023} for more references.

\section{Basic Chord Data}\label{sec:jacobi-section}

Let $\Gamma$ be a closed embedded \emph{smooth} ($\C^\infty$)
hypersurface in $M$. Then
$\Gamma$ bounds a unique domain $\Omega$ with compact closure,
and we let $\nu$ be the outward unit normal of $\Gamma$ with respect to $\Omega$. Any two points of $\Gamma$ are joined
by a unique geodesic, which we call a \emph{chord}. We begin by recording some basic facts about these chords.

\subsection{Incidence cosines}\label{sec:chords}
Fix distinct points $p,q\in\Gamma$, let $\gamma\colon[0,r]\to M$
be the unit speed geodesic from $p$ to $q$, and set
$$
r(p,q)\coloneqq\dist(p,q),\qquad
u\coloneqq\langle\gamma'(0),-\nu_p\rangle,\qquad
v\coloneqq\langle\gamma'(r),\nu_q\rangle,
$$
where $\nu_p\coloneqq\nu(p)$ and $\nu_q\coloneqq\nu(q)$, so that
$-1\leq u,v\leq1$. Thus $u$ is the cosine of the angle between
the directed chord $pq$ and the inward normal at $p$, and $v$ is
the cosine of the angle between $pq$ and the outward normal at
$q$; we call $u$ and $v$ the \emph{incidence cosines} of the
chord.

The distance function
$\dist\colon M\times M\to\R$ is smooth off the diagonal, since $M$ is
Cartan--Hadamard, and we write $r_q\coloneqq\dist(\cdot,q)$,
$r_p\coloneqq\dist(p,\cdot)$. By the first-variation formula,
$\nabla r_q(p)=-\gamma'(0)$ and $\nabla r_p(q)=\gamma'(r)$, so
\begin{equation}\label{eq:u-v}
u=\langle\nabla r_q(p),\nu_p\rangle,
\qquad v=\langle\nabla r_p(q),\nu_q\rangle .
\end{equation}
The tangential gradients are
$\nabla_\Gamma r_q=\nabla r_q-u\nu_p$ and
$\nabla_\Gamma r_p=\nabla r_p-v\nu_q$, and hence
\begin{equation}\label{eq:tangential-gradients}
|\nabla_\Gamma r_q|^2=1-u^2,
\qquad
|\nabla_\Gamma r_p|^2=1-v^2.
\end{equation}

\subsection{Jacobi fields}\label{sec:jacobi-fields}
Along $\gamma$ fix an orthonormal frame $E_1(t),\dots,E_{n-1}(t)$
of $\gamma'(t)^\perp$ which is parallel in $t$; through it, every
normal space $\gamma'(t)^\perp$ is identified with $\R^{n-1}$. With
$R$ the Riemann curvature tensor of $M$, let
$$
K_{ij}(t)
\coloneqq\Big\langle
R\big(E_j(t),\gamma'(t)\big)\gamma'(t),\,E_i(t)\Big\rangle.
$$
This $(n-1)\times(n-1)$ matrix is symmetric, and it is
nonpositive, because for $X=\sum_i\xi^iE_i(t)$, we have
$
\xi^TK(t)\xi=\langle R(X,\gamma')\gamma',X\rangle
=\operatorname{sec}_M(X\wedge\gamma')|X|^2\leq0.
$
We write $A\geq B$ for symmetric matrices if $A-B$
is positive semidefinite. So $K\leq0$.

A vector field $Z$ along $\gamma$ with $Z(t)\perp\gamma'(t)$ for
all $t$ becomes, in the frame, a function
$Z\colon[0,r]\to\R^{n-1}$, and the Jacobi equation
$Z''+R(Z,\gamma')\gamma'=0$ becomes
$
Z''+KZ=0.
$
The Jacobi fields normal to $\gamma$ which vanish at $p$ are
therefore $Z(t)=A(t)Z'(0)$, with
\begin{equation}\label{eq:angular-jacobi}
A''(t)+K(t)A(t)=0,\qquad A(0)=0,\qquad A'(0)=I_{n-1}.
\end{equation}

Let $\Sph_p$ be the unit sphere in $T_pM$ and $\theta\in\Sph_p$.
Varying the direction $\theta$ of the geodesic $\exp_p(t\theta)$
in a direction $\xi\in T_\theta\Sph_p=\theta^\perp$ produces the
Jacobi field $Z_\xi$ with $Z_\xi(0)=0$ and $Z_\xi'(0)=\xi$, which
in the frame is $A(t)\xi$. Consequently the volume element of $M$
in geodesic polar coordinates about $p$ is
$
dV=\det A(t)\,dt\,d\theta,
$
and we call
$$
J(p,q)\coloneqq\det A\big(r(p,q)\big)
$$
the \emph{polar Jacobian}: it is the Jacobian, at
$\theta=\gamma'(0)$, of the map $\theta\mapsto\exp_p(r\theta)$
from $\Sph_p$ onto the geodesic sphere of radius $r=r(p,q)$ about
$p$; in $\R^n$ it equals $r^{n-1}$. As the curvature of $M$ is
nonpositive, Rauch's comparison theorem gives
\begin{equation}\label{eq:B-lower-bound}
|A(t)\xi|\geq t|\xi|,
\qquad
J\geq r^{n-1}.
\end{equation}

\subsection{Derivative of the incidence cosines}\label{sec:angles}
When $q$ is displaced by a vector $Y\in T_qM$ orthogonal to the
chord $pq$, the direction of the chord at $p$ turns by $A(r)^{-1}Y$,
with $Y$ expressed in the frame $E_i(r)$, since $A(r)^{-1}$ is
the derivative of the radial projection from $p$; in $\R^n$ this
is $Y/r$. Accordingly we set
\begin{equation}\label{eq:W}
\tilde A\coloneqq rA(r)^{-T},\qquad\text{so that}\qquad
\det \tilde A=\frac{r^{n-1}}{J},
\end{equation}
and $\|\tilde A\|\leq1$ by \eqref{eq:B-lower-bound}, where $\|\cdot\|$
is the operator norm. The polar Jacobian is symmetric, $J(p,q)=J(q,p)$
\cite[Lem.~5]{yau1975}; under reversal of the chord, $\tilde A$ is
replaced by $\tilde A^T$ in the parallel identifications.
Put
$$
\tilde\nu_p\coloneqq\nu_p-u\nabla r_q(p),\qquad
\tilde\nu_q\coloneqq\nu_q-v\nabla r_p(q),\qquad
w\coloneqq\langle \tilde A\tilde\nu_p,\tilde\nu_q\rangle .
$$
Thus $w$ is the inner product of the components $\tilde\nu_p$,
$\tilde\nu_q$ of the normals transverse to the chord, after
$\tilde\nu_p$ is carried to $q$ by $\tilde A$. Since
$\tilde A$ is replaced by $\tilde A^T$ under reversal of the
chord, $w(p,q)=w(q,p)$. Since $|\tilde\nu_p|^2=1-u^2$,
$|\tilde\nu_q|^2=1-v^2$, and $\|\tilde A\|\leq1$, the
Cauchy--Schwarz inequality gives
\begin{equation}\label{eq:zeta-bound-flat}
|w|\leq h\coloneqq\sqrt{(1-u^2)(1-v^2)}
\leq h_0\coloneqq1-\frac{u^2+v^2}{2}.
\end{equation}

\begin{lemma}\label{lem:angular}
With $p$ fixed and $q$ varying on $\Gamma\setminus\{p\}$,
\begin{equation*}
\nabla_{\Gamma}u=-\frac1r(\tilde A\tilde\nu_p)^\top,\qquad
\left\langle\nabla_{\Gamma}u,\nabla_{\Gamma}r_p\right\rangle=\frac vrw .
\end{equation*}
With
$q$ fixed and $p$ varying,
$
\langle\nabla_\Gamma v,\nabla_\Gamma r_q\rangle=\frac urw.
$
\end{lemma}

\begin{proof}
Let $P_p(q)\coloneqq\exp_p^{-1}(q)/r$. For an endpoint variation $Y\in
T_qM$, the corresponding angular Jacobi field gives
$
dP_p(Y)=A(r)^{-1}Y^\perp,
$
where $Y^\perp$ is the component of $Y$ orthogonal to $\gamma'(r)$.
Since $u=-\langle P_p(q),\nu_p\rangle$, differentiation gives
$d_qu(Y)=-\langle A(r)^{-T}\tilde\nu_p,Y\rangle$. Restricting to
$T_q\Gamma$ proves the first formula; the second follows from
\eqref{eq:tangential-gradients}, since $\nabla_\Gamma r_p=\nabla
r_p(q)-v\nu_q$ and $\tilde A\tilde\nu_p\perp\gamma'(r)$. The last
assertion follows by reversing the chord, which replaces $A(r)$ by
$A(r)^T$.
\end{proof}

\section{Endpoint Comparison I}\label{sec:comparisons}\label{subsec:jacobi}

The curvature of $M$ affects the chords of $\Gamma$ through the polar
Jacobian $J$, which exceeds $r^{n-1}$ by \eqref{eq:B-lower-bound},
and the Hessian of $r$ at the end points $p$ and $q$. Here we compare these two effects. On $T_pM$ and $T_qM$ let
\begin{equation}\label{eq:hessian-remainders}
\cE_p\coloneqq\nabla^2r_q(p)-\frac{g-dr_q\otimes dr_q}{r},
\qquad
\cE_q\coloneqq\nabla^2r_p(q)-\frac{g-dr_p\otimes dr_p}{r},
\end{equation}
with $g$ the metric of $M$; these symmetric forms vanish
identically in $\R^n$, where
$\nabla^2r_q=\frac1r(g-dr_q\otimes dr_q)$. Since
$\nabla r_q(p)=-\gamma'(0)$ lies in the kernel of both terms of
$\cE_p$, the form $\cE_p$ is determined by its restriction to
$\gamma'(0)^\perp$; we write $\bcE_p$ for the matrix of that
restriction in the frame $E_i(0)$, which is the matrix of
$\nabla^2r_q(p)$ on $\gamma'(0)^\perp$ minus $I_{n-1}/r$, and
likewise $\bcE_q$ for the matrix of $\cE_q$ on $\gamma'(r)^\perp$.
Hessian comparison gives $\cE_p,\cE_q\geq0$. In the frame
$E_i(t)$, the restriction of $\nabla^2r_p$ to $\gamma'(t)^\perp$
at $\gamma(t)$ is $A'(t)A(t)^{-1}$, and the same holds along the
reversed chord; thus
\begin{equation}\label{eq:E-from-A}
\bcE_q=A'(r)A(r)^{-1}-r^{-1}I_{n-1},
\end{equation}
and likewise for $\bcE_p$. The following estimate shows how
$\cE_p$ and $\cE_q$ control the excess of $J$ over $r^{n-1}$.

\begin{proposition}[Endpoint comparison I]\label{lem:jacobi-endpoint}
For every distinct pair of points $p,q\in\Gamma$ and all $\lambda,\mu>0$,
\begin{equation}\label{eq:weighted-jacobi}
\lambda\tr_\Gamma\cE_q+\mu\tr_\Gamma\cE_p
\geq\frac{2|uv|\sqrt{\lambda\mu}}{r}\Big(1-\frac{r^{n-1}}{J}\Big)
\geq\frac{4|uv|\lambda\mu}{(\lambda+\mu)r}\Big(1-\frac{r^{n-1}}{J}\Big).
\end{equation}
In particular,
$
\tr_\Gamma\cE_p+\tr_\Gamma\cE_q\geq\frac{2|uv|}{r}\big(1-\frac{r^{n-1}}{J}\big).
$
\end{proposition}

\begin{proof}
Work along the unit-speed geodesic $\gamma\colon[0,r]\to M$ from
$p$ to $q$, in the parallel frame of
Section~\ref{sec:jacobi-fields}, with curvature matrix $K(t)$ and
Jacobi matrix $A(t)$; both sides of the inequality will be
expressed through the second derivatives of the distance at the
endpoints of $\gamma$. A Jacobi field $Z$ perpendicular to
$\gamma$, viewed as a map $Z\colon[0,r]\to\R^{n-1}$, satisfies
$
Z''+KZ=0.
$
Denote by $T$ the solution operator of this equation, that is, the
matrix for which
$$
\begin{pmatrix}Z(r)\\ Z'(r)\end{pmatrix}
=\begin{pmatrix}T_{11}&T_{12}\\ T_{21}&T_{22}\end{pmatrix}
\begin{pmatrix}Z(0)\\ Z'(0)\end{pmatrix}.
$$
Since the Wronskian
$
Z_1^TZ_2'-(Z_1')^TZ_2
$
of two solutions $Z_1,Z_2$ is constant,
\begin{equation}\label{eq:symplectic-relations}
T_{11}T_{12}^T=T_{12}T_{11}^T,
\qquad T_{12}^TT_{22}=T_{22}^TT_{12},
\qquad T_{11}^TT_{22}-T_{21}^TT_{12}=I.
\end{equation}
Moreover $T_{12}=A(r)$ and $T_{22}=A'(r)$, and $T_{12}$ is
invertible since $\gamma$ has no conjugate points. Given
$x,y\in\R^{n-1}$, let $Z\coloneqq Z_{x,y}$ be the transverse
Jacobi field with
$
Z(0)=x
$
and
$
Z(r)=y.
$
Solving the transfer relation with the help of
\eqref{eq:symplectic-relations} gives
$$
\begin{pmatrix}-Z'(0)\\ Z'(r)\end{pmatrix}
=Q\begin{pmatrix}x\\y\end{pmatrix},
\qquad
Q\coloneqq
\begin{pmatrix}
T_{12}^{-1}T_{11}&-T_{12}^{-1}\\
-T_{12}^{-T}&T_{22}T_{12}^{-1}
\end{pmatrix},
$$
whose diagonal blocks are symmetric by the first two relations in
\eqref{eq:symplectic-relations}; we call $Q$ the \emph{endpoint
matrix} of the chord, as it sends the endpoint values of a
transverse Jacobi field to its outward endpoint derivatives. An integration by parts along
$\gamma$ yields
\begin{equation}\label{eq:endpoint-energy}
-\langle x,Z'(0)\rangle+\langle y,Z'(r)\rangle
=\int_0^r\frac d{dt}\langle Z,Z'\rangle\,dt
=\int_0^r\big(|Z'|^2-\langle KZ,Z\rangle\big)\,dt.
\end{equation}
By the second-variation formula for the distance between the
endpoints, or directly from \eqref{eq:endpoint-energy}, first with
$y=0$ and then with $x=0$, the diagonal blocks of $Q$ are the
restrictions of $\nabla^2r_q(p)$ to $\gamma'(0)^\perp$ and of
$\nabla^2r_p(q)$ to $\gamma'(r)^\perp$. Hence, in terms of the
matrices $\bcE_p$, $\bcE_q$ of the remainders
\eqref{eq:hessian-remainders},
\begin{equation}\label{eq:Ep-transfer}
\bcE_p=T_{12}^{-1}T_{11}-r^{-1}I,
\qquad
\bcE_q=T_{22}T_{12}^{-1}-r^{-1}I;
\end{equation}
in particular $\bcE_q=A'(r)A(r)^{-1}-r^{-1}I$, which is
\eqref{eq:E-from-A}.
For arbitrary endpoint values,
$$
\int_0^r|Z'|^2\,dt
\geq\frac1r\left|\int_0^rZ'\,dt\right|^2
=\frac{|y-x|^2}{r},
$$
and $-\langle KZ,Z\rangle\geq0$ since $K\leq0$. Therefore
$
Q\geq Q_0
\coloneqq\frac1r\left(\begin{smallmatrix}I&-I\\-I&I\end{smallmatrix}\right).
$
Since $T_{12}=A(r)$, the matrix $\tilde A$ of \eqref{eq:W} is $rT_{12}^{-T}$.
Using \eqref{eq:Ep-transfer}, the matrix inequality
$r(Q-Q_0)\geq0$ becomes
\begin{equation}\label{eq:block-positive}
\begin{pmatrix}
r\bcE_p&(I-\tilde A)^T\\
I-\tilde A&r\bcE_q
\end{pmatrix}\geq0.
\end{equation}

To extract the Jacobian defect from \eqref{eq:block-positive},
put $\alpha\coloneqq\tr \bcE_p$, $\beta\coloneqq\tr \bcE_q$, and $N\coloneqq I-\tilde A$.
Let $N=\sum_i s_i\,\zeta_i\xi_i^{T}$ be a singular value
decomposition of $N$, with $s_i\geq0$ and $\{\xi_i\}$,
$\{\zeta_i\}$ orthonormal bases of $\R^{n-1}$, so that
$N\xi_i=s_i\zeta_i$. By \eqref{eq:block-positive},
$
s_i^2\leq r^2\langle \bcE_p\xi_i,\xi_i\rangle
\langle \bcE_q\zeta_i,\zeta_i\rangle.
$
Hence, by the Cauchy--Schwarz inequality, the trace norm
$\|\cdot\|_{S_1}$, i.e., the sum of the singular values, satisfies
$$
\|I-\tilde A\|_{S_1}
=\sum_{i=1}^{n-1}s_i
\leq r
\sqrt{\sum_{i=1}^{n-1}\langle \bcE_p\xi_i,\xi_i\rangle}
\sqrt{\sum_{i=1}^{n-1}\langle \bcE_q\zeta_i,\zeta_i\rangle}
=r\sqrt{\alpha\beta}.
$$
By \eqref{eq:B-lower-bound}, every singular value of $A(r)=T_{12}$ is
at least $r$. Hence the singular values $s_i(\tilde A)$ lie in $(0,1]$.
Moreover,
$
\det \tilde A=r^{n-1}/\det A(r)=r^{n-1}/J.
$
For $0\leq s_i\leq1$,
we have
$
1-\prod_{i=1}^{n-1}s_i
\leq\sum_{i=1}^{n-1}(1-s_i).
$
It follows that
\begin{equation}\label{eq:det-trace-norm}
1-\frac{r^{n-1}}{J}
\leq\sum_{i=1}^{n-1}(1-s_i(\tilde A))
=\|I\|_{S_1}-\|\tilde A\|_{S_1}
\leq\|I-\tilde A\|_{S_1}
\leq r\sqrt{\alpha\beta}.
\end{equation}

It remains to pass from $\tr\bcE_p$, $\tr\bcE_q$ to the traces
over $T_p\Gamma$, $T_q\Gamma$. Write
$
\nu_p=u\nabla r_q+\tilde\nu_p,
$
$
|\tilde\nu_p|^2=1-u^2.
$
Because $\cE_p$ vanishes in the radial direction,
$
\tr_\Gamma\cE_p=\alpha-\langle\bcE_p\tilde\nu_p,\tilde\nu_p\rangle.
$
Since $\bcE_p\geq0$,
$
\langle\bcE_p\tilde\nu_p,\tilde\nu_p\rangle
\leq|\tilde\nu_p|^2\tr \bcE_p=(1-u^2)\alpha,
$
and therefore
$
\tr_\Gamma\cE_p\geq u^2\alpha.
$
The same argument gives $\tr_\Gamma\cE_q\geq v^2\beta$. Combining these equations with
\eqref{eq:det-trace-norm},  and the
arithmetic--geometric mean inequality, yields
$$
\lambda\tr_\Gamma\cE_q+\mu\tr_\Gamma\cE_p\geq \lambda v^2\beta+\mu u^2\alpha
\geq2|uv|\sqrt{\lambda\mu\alpha\beta}
\geq\frac{2|uv|\sqrt{\lambda\mu}}{r}\Big(1-\frac{r^{n-1}}{J}\Big),
$$
which is the first inequality in \eqref{eq:weighted-jacobi}; the second is $\sqrt{\lambda\mu}\geq2\lambda\mu/(\lambda+\mu)$.
\end{proof}

\section{Endpoint Comparison II}\label{subsec:joint}

To establish our main results in dimensions $8$ and $9$ we need to sharpen the estimate of the last section.
Chords of $\Gamma$ carry three pieces of information about the
curvature of $M$: (i) how much the geodesic spheres about one endpoint
have spread by the time they reach the other, which is the polar
Jacobian $J$; (ii) how much the geodesic sphere about one endpoint
curves at the other beyond the $1/r$ due to its radius, which is
the Hessian remainder $\cE$ there; and (iii) how much the direction of
the chord at one endpoint turns when the other endpoint is moved,
which is measured by the map $\tilde A$ of \eqref{eq:W}.

These three effects
cannot vary independently. Proposition~\ref{lem:jacobi-endpoint}
described one relation: the spreading in (i) is bounded above by
the extra curvature in (ii). The map $\tilde A$ in (iii) is
subject to the same constraint: its deviation from the identity
is bounded above by the extra curvature in (ii), as we show below.
We generalize Proposition~\ref{lem:jacobi-endpoint}
to a single estimate which bounds the spreading in (i) above by
any nonnegative combination of the extra curvature in (ii) and
the deviation of $\tilde A$ in (iii), with a sharp constant.
To formulate the estimate, let
\begin{equation}\label{eq:M-block}
\M\coloneqq\begin{pmatrix} r\bcE_p&(I-\tilde A)^T\\ I-\tilde A&r\bcE_q\end{pmatrix}
\end{equation}
which is positive semidefinite by \eqref{eq:block-positive}.
Let $\CC\coloneqq\left(\begin{smallmatrix}U&W\\W^T&V\end{smallmatrix}\right)$ be a positive semidefinite symmetric  $2(n-1)\times2(n-1)$
matrix, written in $(n-1)\times(n-1)$ blocks, and put
\begin{equation}\label{eq:lambda-C}
\lambda(\CC)\coloneqq\min_{|\xi|=1}
\Big\{2\sqrt{\langle U\xi,\xi\rangle\langle V\xi,\xi\rangle}
+2\langle W\xi,\xi\rangle\Big\}\geq 0.
\end{equation}

\begin{proposition}[Endpoint comparison II]\label{thm:joint}
For every positive semidefinite symmetric matrix $\CC$,
\begin{equation}\label{eq:joint}
\tr(\CC\M)\geq\lambda(\CC)\Big(1-\frac{r^{n-1}}{J}\Big).
\end{equation}
\end{proposition}

Proposition~\ref{lem:jacobi-endpoint} follows by taking
$W=0$, $U=\mu(I-\tilde\nu_p\tilde\nu_p^T)$, and
$V=\lambda(I-\tilde\nu_q\tilde\nu_q^T)$, since then
$\tr(\CC\M)=r\big(\mu\tr_\Gamma\cE_p+\lambda\tr_\Gamma\cE_q\big)$,
as in the proof of Proposition~\ref{lem:jacobi-endpoint}, and
$\lambda(\CC)\geq2|uv|\sqrt{\lambda\mu}$.

To prove Proposition~\ref{thm:joint}, we
rescale the chord to unit length, so that $r$ drops out and both
sides of \eqref{eq:joint} become functions of a single nonnegative
curvature matrix $\tilde K$ on $[0,1]$. Put
$\tilde K(t)\coloneqq-r^2K(rt)\geq0$ and $Y(t)\coloneqq A(rt)/r$
for $0\leq t\leq1$, so that
\begin{equation}\label{eq:normalized-jacobi}
Y''=\tilde K Y,\qquad Y(0)=0,\qquad Y'(0)=I,
\end{equation}
and $\tilde J\coloneqq\det Y(1)=J/r^{n-1}\geq1$. A transverse Jacobi
field $Z$ along $\gamma$ corresponds to $\tilde Z(t)=Z(rt)$, with
$\tilde Z'=rZ'$, so the endpoint matrix $Q$ of
Section~\ref{subsec:jacobi} becomes $rQ$ for
\eqref{eq:normalized-jacobi}, and by
\eqref{eq:Ep-transfer}--\eqref{eq:block-positive}
\begin{equation}\label{eq:M-normalized}
\M=rQ-\begin{pmatrix}I&-I\\-I&I\end{pmatrix}.
\end{equation}

First we consider the case where $\tilde K$ consists of finitely many point masses:
\begin{equation}\label{eq:point-potential}
\tilde K=\sum_{i=1}^k\rho_i\,\xi_i\xi_i^T\delta_{t_i},
\qquad \rho_i>0,\quad|\xi_i|=1,\quad0<t_1<\dots<t_k<1.
\end{equation}
Here $\delta_{t_i}$ is the Dirac mass at $t_i$, and $\xi_i\xi_i^T$
is the orthogonal projection onto the line spanned by $\xi_i$; so
$\langle\tilde KX,X\rangle=\rho_i\langle X,\xi_i\rangle^2\delta_{t_i}$
near $t_i$, that is, the plane spanned by $\gamma'$ and a
transverse vector $X$ is curved at $t_i$ in proportion to
$\langle X,\xi_i\rangle^2$, and not at all when $X\perp\xi_i$.
Accordingly, solutions of $Y''=\tilde KY$ are continuous and
affine between the nodes, and $Y'$ jumps at $t_i$ by
$\rho_i\xi_i\langle\xi_i,Y(t_i)\rangle$, so only the
$\xi_i$-component of a Jacobi field is affected.
Let $B$ be the $k\times2(n-1)$ matrix with rows
$B_i\coloneqq\big((1-t_i)\xi_i^T,\ t_i\xi_i^T\big)$, let
$$
G_{ij}\coloneqq\big(\min\{t_i,t_j\}-t_it_j\big)\langle\xi_i,\xi_j\rangle,
\qquad D_\rho\coloneqq\operatorname{diag}(\rho_1,\dots,\rho_k),
$$
and note that $\min\{s,t\}-st$ is the Green function of $-d^2/dt^2$
on $[0,1]$ with Dirichlet boundary conditions. Thus $G$ is the
Gram matrix, for the inner product $\int\langle V',V'\rangle$ on
$H^1_0((0,1);\R^{n-1})$, of the functionals $V\mapsto\langle\xi_i,V(t_i)\rangle$,
which are linearly independent, so $G$ is positive definite. We need the following linear algebra fact, which  reduces the required estimate for several point masses to bounds associated with individual masses:

\begin{lemma}\label{lem:matrix-inequality}
Let $G$ be a positive definite $k\times k$ matrix, $D$ a positive
diagonal $k\times k$ matrix, and $B$ a $k\times m$ matrix with rows
$B_1,\dots,B_k$. Then there are $\theta_1,\dots,\theta_k\geq0$
with $\sum_i\theta_i=1$ such that
\begin{equation*}
B^T(D^{-1}+G)^{-1}B
\geq\Big(1-\frac1{\det(I+DG)}\Big)\sum_{i=1}^k\theta_i\frac{B_i^TB_i}{G_{ii}} .
\end{equation*}
\end{lemma}

\begin{proof}
Write $\Phi\coloneqq B^T(D^{-1}+G)^{-1}B$ and
$D=\operatorname{diag}(\rho_1,\dots,\rho_k)$. We first expand
$\det(I+DG)$ in principal minors:
$$
\det(I+DG)=1+\sum_{S\neq\varnothing}m_S,
\qquad
m_S\coloneqq\Big(\prod_{i\in S}\rho_i\Big)\det G_S>0,
$$
where $S$ runs over the nonempty subsets of $\{1,\dots,k\}$ and
$G_S$ is the submatrix of $G$ with rows and columns in $S$.
Next we expand $\det(I+DG)\,\Phi$ in the same way: for every
symmetric $\CC$, $\tr(\CC\Phi)$ is the derivative at $\epsilon=0$
of $\log\det\big(I+D(G+\epsilon B\CC B^T)\big)$, and computing
this derivative once by Jacobi's formula and once minor by minor
gives
\begin{equation*}
\det(I+DG)\,\Phi=\sum_{S\neq\varnothing}m_S\Theta_S,
\qquad
\Theta_S\coloneqq B_S^TG_S^{-1}B_S ,
\end{equation*}
with $B_S$ the rows of $B$ in $S$. Thus
$\Phi/\big(1-\det(I+DG)^{-1}\big)$ is a weighted average of the
matrices $\Theta_S$, and it suffices to show that each $\Theta_S$
dominates $B_i^TB_i/G_{ii}$ for some $i$; the $\theta_i$ are then
the total weights of the sets $S$ assigned to $i$.
In fact, for every $i\in S$, $\Theta_S\geq\Theta_{\{i\}}$.
Indeed, reorder $i$ first and write
$G_S=\left(\begin{smallmatrix}g&z^T\\ z&G'\end{smallmatrix}\right)$,
$B_S=\left(\begin{smallmatrix}B_i\\ B'\end{smallmatrix}\right)$,
with $g=G_{ii}$; then
$$
\Theta_S=\frac{B_i^TB_i}{g}
+\big(B'-zB_i/g\big)^T\big(G'-zz^T/g\big)^{-1}\big(B'-zB_i/g\big)
\geq\frac{B_i^TB_i}{g}=\Theta_{\{i\}},
$$
since $G'-zz^T/g$ is positive definite.
\end{proof}

\begin{lemma}\label{lem:point-curvature}
If $\tilde K$ is a finite sum of point masses given by \eqref{eq:point-potential}, then
\eqref{eq:joint} holds.
\end{lemma}

\begin{proof}
We first show that
\begin{equation}\label{eq:finite-identities}
\M=B^T\big(D_\rho^{-1}+G\big)^{-1}B,
\qquad
\tilde J=\det\big(I+D_\rho G\big).
\end{equation}
By \eqref{eq:endpoint-energy}, $\langle e,rQe\rangle$ is the
energy $\int_0^1\big(|V'|^2+\langle\tilde KV,V\rangle\big)\,dt$ of
the Jacobi field with endpoint values $e=(x,y)$, and it is the
least energy of a field $V$ on $[0,1]$ with those endpoint
values: the difference of $V$ and the Jacobi field lies in
$H^1_0$, the cross term in the energy vanishes by an integration
by parts and \eqref{eq:normalized-jacobi}, and the energy of the
difference is nonnegative since $\tilde K\geq0$. For
\eqref{eq:point-potential} the energy is
$\int_0^1|V'|^2\,dt+\sum_i\rho_i\langle\xi_i,V(t_i)\rangle^2$.
Write such a field as
$(1-t)x+ty+V_0(t)$ with $V_0\in H^1_0$; the cross term
$\int\langle y-x,V_0'\rangle\,dt$ vanishes. If
$\eta_i\coloneqq\langle\xi_i,V_0(t_i)\rangle$, the least value of
$\int|V_0'|^2$ is $\eta^TG^{-1}\eta$, so the energy of the
minimizing field is
$$
|y-x|^2+\min_{\eta\in\R^k}\Big\{\eta^TG^{-1}\eta
+(Be+\eta)^TD_\rho(Be+\eta)\Big\}
=|y-x|^2+e^TB^T\big(D_\rho^{-1}+G\big)^{-1}Be,
$$
by the Woodbury identity
$D_\rho-D_\rho(G^{-1}+D_\rho)^{-1}D_\rho=(D_\rho^{-1}+G)^{-1}$.
Subtracting the Euclidean form $|y-x|^2$ gives the first
identity, by \eqref{eq:M-normalized}.
For the second, integrate $Y''=\tilde KY$ with $Y(0)=0$, $Y'(0)=I$:
$$
Y(t)=tI+\sum_i\rho_i(t-t_i)_+\,\xi_i\xi_i^TY(t_i).
$$
Let $N$ be the $k\times(n-1)$ matrix with rows
$N_i\coloneqq\xi_i^TY(t_i)$, and let $P$, $\bar P$, and $L$ be the
matrices with rows $P_i\coloneqq(1-t_i)\xi_i^T$,
$\bar P_i\coloneqq t_i\xi_i^T$, and entries
$L_{ji}\coloneqq(t_j-t_i)_+\langle\xi_j,\xi_i\rangle$, so that $L$
is strictly lower triangular. Evaluating the display at $t=1$ and
at $t=t_j$, and multiplying the latter by $\xi_j^T$, gives
$$
Y(1)=I+P^TD_\rho N,\qquad (I-LD_\rho)N=\bar P .
$$
Since $\det(I+AB)=\det(I+BA)$ and $\det(I-LD_\rho)=1$,
$$
\det Y(1)=\det\big(I+(I-LD_\rho)^{-1}\bar PP^TD_\rho\big)
=\det\big(I+(\bar PP^T-L)D_\rho\big)=\det(I+GD_\rho),
$$
because $(\bar PP^T-L)_{ji}
=\big(t_j(1-t_i)-(t_j-t_i)_+\big)\langle\xi_j,\xi_i\rangle
=G_{ji}$, as one checks separately for $t_j\geq t_i$ and
$t_j<t_i$. As $\det(I+GD_\rho)=\det(I+D_\rho G)$, this is the
second identity.
By \eqref{eq:finite-identities}, Lemma~\ref{lem:matrix-inequality}
applies with $D=D_\rho$; since $G_{ii}=t_i(1-t_i)$, it gives
$$
\M\geq\big(1-\tilde J^{-1}\big)\sum_i\theta_i\,e_{t_i,\xi_i}e_{t_i,\xi_i}^T,
\qquad
e_{t,\xi}\coloneqq\frac{\big((1-t)\xi,\ t\xi\big)}{\sqrt{t(1-t)}},
$$
for some $\theta_i\geq0$ with $\sum_i\theta_i=1$. Taking the
trace against $\CC$ gives
$$
\tr(\CC\M)\geq\big(1-\tilde J^{-1}\big)\sum_i\theta_i
\langle\CC e_{t_i,\xi_i},e_{t_i,\xi_i}\rangle,
$$
and the quadratic form of $\CC$ at $e_{t,\xi}$ is
\begin{equation}\label{eq:e-t-xi}
\langle\CC e_{t,\xi},e_{t,\xi}\rangle
=\frac{1-t}{t}\langle U\xi,\xi\rangle
+\frac{t}{1-t}\langle V\xi,\xi\rangle+2\langle W\xi,\xi\rangle
\geq\lambda(\CC)
\end{equation}
by the arithmetic--geometric mean inequality. Hence
$\tr(\CC\M)\geq\lambda(\CC)(1-\tilde J^{-1})$, which is
\eqref{eq:joint}.
\end{proof}

\begin{proof}[Proof of Proposition~\ref{thm:joint}]
By Lemma~\ref{lem:point-curvature}, \eqref{eq:joint} holds when
$\tilde K$ is a finite sum of point masses. We approximate the given continuous
$\tilde K\geq0$ by such sums.
Partition $[0,1]$ into intervals $I_j$ of lengths
$h_j\leq\epsilon$, and choose $s_j\in I_j$. On $I_j$, approximate
$\tilde K\,dt$ by rank-one point masses at distinct interior
points of $I_j$, one for each eigenvector of an orthonormal
eigenbasis of $\tilde K(s_j)$ with positive eigenvalue, with
strength $h_j$ times that eigenvalue.
Let $\tilde K_\epsilon$ denote the resulting atomic measure.
Write the Jacobi equation as the first-order system
$$
\frac d{dt}\binom{Y}{Y'}=
\begin{pmatrix}0&I\\ \tilde K(t)&0\end{pmatrix}
\binom{Y}{Y'}.
$$
For the continuous equation, the transfer matrix across $I_j$ is
$$
I+h_j
\begin{pmatrix}0&I\\ \tilde K(s_j)&0\end{pmatrix}
+O\big(h_j^2+h_j\omega(\epsilon)\big),
$$
where $\omega(\epsilon)\coloneqq
\sup_{|s-t|\leq\epsilon}\|\tilde K(s)-\tilde K(t)\|\to0$.
For the atomic equation, free propagation through a segment of
length $\ell$ has transfer matrix
$\left(\begin{smallmatrix}I&\ell I\\0&I\end{smallmatrix}\right)$,
while a point mass $\rho\,\xi\xi^T$ has transfer matrix
$\left(\begin{smallmatrix}I&0\\ \rho\,\xi\xi^T&I\end{smallmatrix}\right)$.
The ordered product of these matrices across $I_j$ has the same
first-order term and an $O(h_j^2)$ remainder, since the sum of the
free lengths is $h_j$ and the sum of the point-mass matrices is
$h_j\tilde K(s_j)$.

All partial transfer products are uniformly bounded, since each
factor is $I+O(h_j)$ and $\sum_jh_j=1$. A telescoping comparison
therefore shows that the full transfer matrices differ by
$$
O\!\Big(\sum_jh_j^2+\omega(\epsilon)\sum_jh_j\Big)
=O\big(\epsilon+\omega(\epsilon)\big).
$$
Hence the full transfer matrices converge. In particular
$\tilde J_\epsilon\to\tilde J$, and, since $Y(1)$ is invertible,
the endpoint matrices $Q_\epsilon$ converge to $Q$. Thus
$\M_\epsilon\to\M$, and \eqref{eq:joint} passes to the limit.
\end{proof}

\begin{note}\label{note:joint-sharp}
The constant $\lambda(\CC)$ cannot be improved for
$\tilde K\geq0$. For a single point mass
$\tilde K=\rho\,\xi\xi^T\delta_t$, \eqref{eq:finite-identities}
gives $\M=(1-\tilde J^{-1})e_{t,\xi}e_{t,\xi}^T$, so
$\tr(\CC\M)/(1-\tilde J^{-1})
=\langle\CC e_{t,\xi},e_{t,\xi}\rangle$,
whose infimum over $t$ and $\xi$ is $\lambda(\CC)$
by \eqref{eq:e-t-xi}; smoothing the point mass, with
$(t,\xi)$ near the infimum, gives continuous $\tilde K\geq0$
with this quotient arbitrarily close to $\lambda(\CC)$.
\end{note}

\section{The Two Main Identities}\label{sec:two-point}

Here we establish the Minkowski--Green
identity, through which the mean curvature of $\Gamma$ enters,
and the degree identity, which encodes that $\Gamma$ bounds a
domain. Their Euclidean
versions, which motivate the construction, are in
\cite[\S2]{chen-ghomi-wang}; here we set them up directly in
Cartan--Hadamard manifolds, for more general weights.

\subsection{Estimates near the diagonal}\label{sec:diagonal}
We first record how fast the chord data $u$, $v$, $\cE_p$,
$\cE_q$, $J/r^{n-1}$ approach their limits as $q\to p$. These
rates are what make the singular integrands below locally
integrable, and they are valid at the $\C^{1,1}$ regularity at
which the identities will later be needed. Throughout, $B_\rho(p)$
is the (closed) geodesic ball in $M$ of radius $\rho$ about $p$, and
$|X|$ is the measure of a submanifold $X\subset M$ in its own
dimension.

\begin{lemma}\label{lem:diagonal-estimates}
Let $\Gamma$ be $\C^{1,1}$. There are $r_0,C>0$, depending only
on $\Gamma$ and on the geometry of $M$ near it, such that for
$0<r=\dist(p,q)<r_0$,
$$
|u|+|v|\leq Cr,
\qquad \|\cE_p\|+\|\cE_q\|\leq Cr,
\qquad |J/r^{n-1}-1|\leq Cr^2,
\qquad \|\tilde A-I\|\leq Cr^2,
$$
and $|\Gamma\cap B_\rho(p)|\leq C\rho^{n-1}$ for $0<\rho<r_0$.
\end{lemma}

\begin{proof}
As $\Gamma$ is compact and $\C^{1,1}$, there is $r_0>0$ with the
following property: in normal coordinates about any $p\in\Gamma$
in which $T_p\Gamma=\R^{n-1}\times\{0\}$ and $\nu_p=e_n$, the set
$\Gamma\cap B_{r_0}(p)$ lies in the graph of a $\C^{1,1}$
function $h\colon\{|y|<r_0\}\subset\R^{n-1}\to\R$,
$$
q=(y,h(y)),
\qquad
h(0)=0,\qquad \nabla h(0)=0,
\qquad |h(y)|\leq C|y|^2,
\qquad |\nabla h(y)|\leq C|y|,
$$
with $C$ independent of $p$. In these coordinates the chord
$\gamma$ from $p$ to such a $q$ is the segment $\gamma(t)=tq/r$,
$0\leq t\leq r$, with $r=|q|$, so $\gamma'\equiv q/r$ and
$\langle\gamma'(0),\nu_p\rangle=h(y)/r=O(r)$. The normal $\nu$ is
Lipschitz, $\Gamma$ being $\C^{1,1}$, so $|\nu_q-\nu_p|\leq Cr$,
and the metric coefficients at $q$ are $\delta_{ij}+O(r^2)$;
hence
$\langle\gamma'(r),\nu_q\rangle=\langle\gamma'(0),\nu_p\rangle+O(r)$,
and \eqref{eq:u-v} gives
$
|u|+|v|\leq Cr.
$
Next, $\|K(t)\|\leq C$ along every chord with $r<r_0$, by
compactness of $\Gamma$, and \eqref{eq:angular-jacobi} can be
written as
$$
A(t)=tI_{n-1}-\int_0^t(t-s)K(s)A(s)\,ds,
$$
from which
$
A(t)=tI_{n-1}+O(t^3)
$
and
$
A'(t)=I_{n-1}+O(t^2).
$
Therefore, by \eqref{eq:E-from-A},
$$
\bcE_q=A'(r)A(r)^{-1}-r^{-1}I_{n-1}=O(r),
\qquad
\frac{J}{r^{n-1}}=\frac{\det A(r)}{r^{n-1}}=1+O(r^2),
$$
and $\tilde A=rA(r)^{-T}=I+O(r^2)$; the bound on $\cE_p$ follows
in the same way along the chord from $q$ to $p$. Finally,
$\Gamma\cap B_\rho(p)$ is contained in the graph of $h$ over
$\{|y|\leq\rho\}$, and since $|\nabla h|\leq C\rho$ there, that
graph has area at most $C\rho^{n-1}$.
\end{proof}

Consequently, for every bounded function $\phi$ on $[-1,1]$, the
singular quantities that occur below satisfy
$$
\frac{|v|}{J}=O(r^{2-n}),
\qquad
\frac{|uv|}{J}=O(r^{3-n}),
\qquad
\frac{|v\phi(u)+u\phi(v)|}{J}=O(r^{2-n}),
$$
and the area bound gives
$\int_{\Gamma\cap B_\eps(p)}r^{-\alpha}=O(\eps^{\,n-1-\alpha})$
for $0\leq\alpha<n-1$. Thus each of these quantities is integrable
on $\Gamma$, and every integral over $\Gamma\times\Gamma$ below in
which one of them appears converges absolutely.

\subsection{The Minkowski--Green identity}
\label{sec:green}
The identities of this section will be applied, in
Sections~\ref{sec:singular-section} and~\ref{sec:CMC}, to
hypersurfaces which are in general only of class $\C^{1,1}$, so from
now on $\Gamma$ is allowed to be $\C^{1,1}$. Its normal $\nu$ is
then Lipschitz, so by Rademacher's theorem the mean curvature
$H=\operatorname{div}_\Gamma\nu$ of $\Gamma$, the sum of its
principal curvatures with respect to $\nu$, exists almost
everywhere and is essentially bounded; pointwise identities
involving $H$ are understood almost everywhere. 

For a function
$f$ defined near $\Gamma$, we write $\nabla f$ and $\nabla^2f$
for its ambient gradient and Hessian,
$\nabla_\Gamma f\coloneqq(\nabla f)^\top$ for the tangential
component of $\nabla f$ along $\Gamma$, and $\Delta_\Gamma$ for
the Laplace--Beltrami operator of $\Gamma$, so that
\begin{equation}\label{eq:laplacian-convention}
\Delta_\Gamma f
=
\tr_\Gamma\nabla^2f-H\langle\nabla f,\nu\rangle,
\end{equation}
where $\tr_\Gamma\nabla^2f=\Delta f-\nabla^2f(\nu,\nu)$ is the
trace of $\nabla^2f$ over $T\Gamma$.

The weights we consider are functions $\psi\colon(0,\infty)\times[-1,1]\to\R$ of the length
of a directed chord and of its terminal incidence cosine, of the
form
\begin{equation}\label{eq:angular-potential}
\partial_r\psi(r,v)=r^{2-n}c(r,v),\qquad
c(r,v)\coloneqq a(r)-b(r)v^2,
\end{equation}
where $a,b\in \C^1([0,\infty))$, $a-b>0$, and $b\geq0$; we set
$C_\psi\coloneqq a(0)$.  The
pole at $r=0$ is what produces nonzero integrals whose values do
not depend on $\Gamma$; the pole-free weights $\psi(r)=r^k$,
$k\geq1$, were considered by Banchoff--Pohl \cite[\S5]{BP}. For
fixed $q\in\Gamma$, define the function $\psi_q$ and the tangent
vector field $X_q$ on $\Gamma\setminus\{q\}$ by
$$
\psi_q(p)\coloneqq\psi\big(r(p,q),v(p,q)\big),
\qquad
X_q(p)\coloneqq\partial_r\psi\big(r(p,q),v(p,q)\big)\,\nabla_\Gamma r_q(p),
$$
where $v(p,q)$ is the incidence cosine of the chord $pq$ at $q$,
as in \eqref{eq:u-v}. Thus $X_q$ is the radial part of
$\nabla_\Gamma\psi_q$. Define
\begin{equation}\label{eq:tilde-laplacian-angular}
\tilde{\Delta}_\Gamma\psi_q\coloneqq\diver_\Gamma X_q
\quad\text{on }\Gamma\setminus\{q\};
\end{equation}
for $b=0$ one has $X_q=\nabla_\Gamma\psi_q$, and
$\tilde{\Delta}_\Gamma\psi_q$ is $\Delta_\Gamma\psi_q$ restricted
to $\Gamma\setminus\{q\}$. We use the radial part of
$\nabla_\Gamma\psi_q$ rather than all of it because
$\Delta_\Gamma\psi_q$ contains, in addition, the term
$\diver_\Gamma(\partial_v\psi\nabla_\Gamma v)$, and with it
$|\nabla_\Gamma v|^2$ and $\Delta_\Gamma v$, which are not among
the quantities that the comparison estimates of
Sections~\ref{sec:comparisons} and~\ref{subsec:joint} control;
the divergence of $X_q$ will involve only $r$, $u$, $v$, $w$, $H$,
and the Hessian remainder $\cE_p$, together with the prescribed
weight and its derivatives.

To compute $\tilde{\Delta}_\Gamma\psi_q$, write
\eqref{eq:hessian-remainders} as
$
\nabla^2 r_q(p)=\frac1r(g-dr_q\otimes dr_q)+\cE_p
$
and take the trace over an orthonormal basis of $T_p\Gamma$; by
\eqref{eq:tangential-gradients},
$$
\tr_\Gamma\nabla^2 r_q
=\frac1r\big(n-1-|\nabla_\Gamma r_q|^2\big)+\tr_\Gamma\cE_p
=\frac{n-2+u^2}{r}+\tr_\Gamma\cE_p.
$$
As $\langle\nabla r_q,\nu_p\rangle=u$,
\eqref{eq:laplacian-convention} yields
$\Delta_\Gamma r_q=(n-2+u^2)/r+\tr_\Gamma\cE_p-Hu$, while
Lemma~\ref{lem:angular} yields
$\langle\nabla_\Gamma v,\nabla_\Gamma r_q\rangle=(u/r)w$.
Inserting these in
$$
\diver_\Gamma X_q=\partial_r^2\psi|\nabla_\Gamma r_q|^2
+\partial_r\partial_v\psi\langle\nabla_\Gamma v,\nabla_\Gamma r_q\rangle
+\partial_r\psi\Delta_\Gamma r_q ,
$$
where $\psi$ and its derivatives are evaluated at $(r,v)$, we get
\begin{equation}\label{eq:curved-Ap-general}
\tilde{\Delta}_\Gamma\psi_q
=\partial_r^2\psi(1-u^2)+\frac ur\partial_r\partial_v\psi\,w
+\partial_r\psi\Big(\frac{n-2+u^2}{r}+\tr_\Gamma\cE_p-Hu\Big).
\end{equation}
After multiplication by $r^{n-1}$, and by
\eqref{eq:angular-potential}, this reads
\begin{equation}\label{eq:curved-D}
r^{n-1}\tilde{\Delta}_\Gamma\psi_q
=r\,\partial_r c(r,v)(1-u^2)+\big((n-1)u^2-rHu\big)c(r,v)
+u\,\partial_v c(r,v)w+rc(r,v)\tr_\Gamma\cE_p ,
\end{equation}
and the formula for $\tilde{\Delta}_\Gamma\psi_p$ is the same
with $p,q$ and $u,v$ interchanged, as $w(p,q)=w(q,p)$. Each term
on the right of \eqref{eq:curved-D} is $O(r)$: $u,v=O(r)$ and
$\tr_\Gamma\cE_p=O(r)$ by Lemma~\ref{lem:diagonal-estimates},
$H$ is bounded, and $a,b$ are $\C^1$. Hence
\begin{equation}\label{eq:Ap-bound}
|\tilde{\Delta}_\Gamma\psi_q|\leq Cr^{2-n},
\end{equation}
and, as $|\Gamma\cap B_\eps(q)|\leq C\eps^{n-1}$ by
Lemma~\ref{lem:diagonal-estimates},
$\tilde{\Delta}_\Gamma\psi_q\in L^1(\Gamma)$.

\begin{proposition}[Minkowski--Green identity]\label{lem:green}
Let $\Gamma$ be $\C^{1,1}$ and $\psi$ satisfy
\eqref{eq:angular-potential}. For every fixed $q\in\Gamma$,
\begin{equation}\label{eq:green-mass}
\int_\Gamma \tilde{\Delta}_\Gamma\psi_q=-C_\psi|\Sph^{n-2}|.
\end{equation}
\end{proposition}

\begin{proof}
Assume first that $\Gamma$ is smooth. For small $\eps>0$ the set
$U_\eps\coloneqq\Gamma\setminus B_\eps(q)$ has $\C^1$ boundary,
with outward conormal
$\eta_\eps\coloneqq-\nabla_\Gamma r_q/|\nabla_\Gamma r_q|$. On
$\partial U_\eps$ we have $v=O(\eps)$, and $a$ is $\C^1$, so
\eqref{eq:angular-potential} gives
$$
\langle X_q,\eta_\eps\rangle
=-\partial_r\psi(\eps,v)|\nabla_\Gamma r_q|
=-\big(C_\psi\,\eps^{2-n}+O(\eps^{3-n})\big)|\nabla_\Gamma r_q|.
$$
Here $|\nabla_\Gamma r_q|=1+O(\eps^2)$, and the graph
representation of $\Gamma$ near $q$ from the proof of
Lemma~\ref{lem:diagonal-estimates} gives
$
|\partial U_\eps|=|\Sph^{n-2}|\eps^{n-2}\big(1+O(\eps^2)\big).
$
Hence
\begin{equation*}
\int_{\partial U_\eps}\langle X_q,\eta_\eps\rangle
=-C_\psi|\Sph^{n-2}|+o(1).
\end{equation*}
The divergence theorem on $U_\eps$ gives
$
\int_{U_\eps}\tilde{\Delta}_\Gamma\psi_q
=\int_{\partial U_\eps}\langle X_q,\eta_\eps\rangle,
$
and letting $\eps\to0$, which is permitted since
$\tilde{\Delta}_\Gamma\psi_q\in L^1(\Gamma)$, proves
\eqref{eq:green-mass}.

Now let $\Gamma$ be $\C^{1,1}$, and put $X\coloneqq X_q$ for a
fixed $q\in\Gamma$. This is a locally Lipschitz tangent field on
$\Gamma\setminus\{q\}$, because $v(\cdot,q)$ is smooth away from
$q$ and
$\nabla_\Gamma r_q=\nabla r_q-\langle\nabla r_q,\nu\rangle\nu$ is
Lipschitz. In a local $\C^{1,1}$ parametrization of $\Gamma$, the
metric coefficients $g_{ij}$ are Lipschitz and
$$
\tilde{\Delta}_\Gamma\psi_q=\operatorname{div}_\Gamma X
=\frac{1}{\sqrt{\det g}}\,\partial_i\big(\sqrt{\det g}\,X^i\big),
$$
with $\sqrt{\det g}\,X^i$ locally Lipschitz; as the distributional
derivative of a Lipschitz function is its almost-everywhere
derivative \cite[\S4.2.3]{evans-gariepy2015},
$\tilde{\Delta}_\Gamma\psi_q$ is $L^\infty_{\rm loc}$ on
$\Gamma\setminus\{q\}$. Wherever $\nu$ is differentiable, the
computation of $\operatorname{div}_\Gamma X$ that gives
\eqref{eq:laplacian-convention} for smooth $\Gamma$ goes through
unchanged, so \eqref{eq:laplacian-convention}, and with it
\eqref{eq:curved-Ap-general} and \eqref{eq:Ap-bound}, hold almost
everywhere on $\Gamma\setminus\{q\}$. Lastly, for a Lipschitz
field $X$ on the $\C^{1,1}$ domains $U_\eps$, the divergence
theorem
$\int_{U_\eps}\operatorname{div}_\Gamma X
=\int_{\partial U_\eps}\langle X,\eta_\eps\rangle$ remains valid
\cite[\S4.3]{evans-gariepy2015}, and the preceding flux computation used nothing beyond
Lemma~\ref{lem:diagonal-estimates}, which holds at $\C^{1,1}$
regularity. So the identity holds for $\Gamma$.
\end{proof}

As we pointed out in \cite[\S2.2]{chen-ghomi-wang}, the name of the identity here
refers to Minkowski's formula
$\int_\Gamma H\langle x-q,\nu\rangle=(n-1)|\Gamma|$
\cite{hsiung1954,montiel-ros2009}, which the same computation
produces in $\R^n$ for the pole-free weight $\psi(r)=r^2/2$, and
to the point-mass singularity of $\psi_q$, which is that of a
Green function.

\subsection{The degree identity}\label{sec:degree}
Fix $p\in\Gamma$ and let
$$
P_p(q)\coloneqq\frac{\exp_p^{-1}(q)}{\dist(p,q)}\in\Sph_p
$$
be radial projection from $p$; thus $P_p(q)=\gamma'(0)$ and
$u=\langle P_p(q),-\nu_p\rangle$. Let
$
\Sph_p^-
$
be the open hemisphere centered at $-\nu_p$, so that
$\theta\in\Sph_p^-$ exactly when the geodesic ray
$\exp_p(t\theta)$, $t>0$, initially enters $\Omega$. The points
of $\Gamma$ on this ray form $P_p^{-1}(\theta)$, and at a
transverse intersection $q$ the sign of
$v(p,q)=\langle\gamma'(r),\nu_q\rangle$ tells whether the ray is
leaving $\Omega$ ($v>0$) or entering it ($v<0$). By the area
formula, applied to $\Gamma$ and to $\{\textup{Jac}(P_p)=0\}$,
the intersections are transverse and finite for almost every
$\theta$; and as the ray alternately exits and enters $\Omega$,
\begin{equation*}
\sum_{q\in P_p^{-1}(\theta)}\sgn v(p,q)
=\1_{\Sph_p^-}(\theta),
\end{equation*}
the indicator function of $\Sph_p^-$. This is a special case of
the formula for the winding number of $\Gamma$ about a point,
cf.\ \cite[(3.20)]{hoisington2021}, with the base point on
$\Gamma$. The differential of $P_p$ at $q$ is $A(r)^{-1}$
composed with the orthogonal projection
$T_q\Gamma\to\gamma'(r)^\perp$, whose determinant is
$\langle\nu_q,\gamma'(r)\rangle=v$; hence
$\textup{Jac}(P_p)=v/J$ on $\Gamma\setminus\{p\}$, cf.\
\cite[Cor.~2.13]{hoisington2021}. Let
$\phi\colon[-1,1]\to\R$ be a bounded Borel function. As
$u=\langle P_p(q),-\nu_p\rangle$ is a function of $P_p(q)$
alone, the area formula applied to $\phi(u)$ gives
\begin{equation*}
\int_\Gamma \phi(u)\frac{v}{J}
=\int_{\Sph_p}\phi\big(\langle\theta,-\nu_p\rangle\big)
\sum_{q\in P_p^{-1}(\theta)}\sgn v(p,q)
=\int_{\Sph_p^-}\phi\big(\langle\theta,-\nu_p\rangle\big)
=|\Sph^{n-2}|\,I_\phi ,
\end{equation*}
where
$$
I_\phi\coloneqq\int_0^1\phi(t)(1-t^2)^{(n-3)/2}\,dt,
$$
the integration over $\Gamma$ being in $q$ for fixed $p$, and
the last step being the slicing of the hemisphere by
$t=\langle\theta,-\nu_p\rangle$. When $\Gamma$ is only
$\C^{1,1}$, the area formula is applied on
$\Gamma\setminus B_\eps(p)$, where $P_p$ is Lipschitz, and
$\eps\to0$ is taken using the integrability of $|v|/J$
(Lemma~\ref{lem:diagonal-estimates}).

\begin{proposition}[Degree identity]\label{prop:boundary-integral}
Let $\Gamma$ be $\C^{1,1}$. For every bounded Borel function
$\phi$ on $[-1,1]$,
\begin{equation}\label{eq:boundary-integral}
\int_{\Gamma\times\Gamma}\frac{v\phi(u)+u\phi(v)}{J}
=2 I_\phi|\Sph^{n-2}|\,|\Gamma|.
\end{equation}
\end{proposition}

\begin{proof}
Since the right side of the preceding identity is independent of
$p$, integration over $p\in\Gamma$ gives
$\int_{\Gamma\times\Gamma} \phi(u)v/J=I_\phi|\Sph^{n-2}||\Gamma|$.
Exchanging $p$ and $q$ exchanges $u$ and $v$ and leaves $J$
unchanged, since $J(p,q)=J(q,p)$; so the same holds with
$u\phi(v)$ in place of $\phi(u)v$, and adding the two identities
gives \eqref{eq:boundary-integral}.
\end{proof}

For $\phi(t)=t^{n-2}$ the right
side of \eqref{eq:boundary-integral} is $2^{2-n}|\Sph^{n-1}||\Gamma|$, the form of the degree
identity used in \cite{chen-ghomi-wang}; the freedom in $\phi$ here
will be needed in dimensions $4$ and $6$ through $9$.

\section{Singular Isoperimetric Regions}\label{sec:singular-section}
\label{subsec:singular-removal}

An \emph{isoperimetric region} in a geodesic ball $B\subset M$ is
a set which minimizes perimeter among the
subsets of $B$ for a given volume $0<V<|B|$.
Perimeter is measured in $M$, including the part of the boundary
on $\partial B$. For
$n\geq8$, the boundary of an isoperimetric region may have a singular set of codimension at
least $7$. Here we show that the identities of
Section~\ref{sec:two-point} hold across such a set. 

\begin{proposition}\label{prop:singular-identities}
Let $\Gamma$ be the boundary of an isoperimetric region in a
geodesic ball, and $S$ be its singular set. Then, for every
$\psi$ as in \eqref{eq:angular-potential} and every Lipschitz
function $\phi$ on $[-1,1]$ with $\phi(0)=0$, the identities of
Propositions~\ref{lem:green} and~\ref{prop:boundary-integral}
hold for $\Gamma$: for every $q\in\Gamma\setminus S$,
$$
\int_\Gamma\tilde{\Delta}_\Gamma\psi_q=-C_\psi|\Sph^{n-2}|,
\qquad\text{and}\qquad
\int_{\Gamma\times\Gamma}\frac{v\phi(u)+u\phi(v)}{J}
=2I_\phi|\Sph^{n-2}|\,|\Gamma|,
$$
and both integrands are absolutely integrable on $\Gamma\times\Gamma$.
\end{proposition}

By \cite[Lem.~7.2(i)--(ii)]{ghomi-spruck2022}, using the
standard regularity theory \cite{gmt1983,Stredulinsky-Ziemer1997,morgan2003}, we know that
$S$ is compact,
disjoint from $\partial B$, and of 
dimension at most $n-8$, so $\mathcal H^{n-2}(S)=0$, where
$\mathcal H^k$ denotes $k$-dimensional Hausdorff measure; and
$\Gamma\setminus S$, the \emph{regular part} of $\Gamma$, is
locally $\C^{1,1}$, with bounded mean curvature $H$. We retain the
notation $\nabla_\Gamma$, $\Delta_\Gamma$, $\diver_\Gamma$, and
$\tr_\Gamma$ for the differential operators on the regular part;
in particular $X_q$ and $\tilde{\Delta}_\Gamma\psi_q=\diver_\Gamma X_q$
are defined on $\Gamma\setminus(S\cup\{q\})$ by
\eqref{eq:tilde-laplacian-angular}. Integrals over $\Gamma$ are
with respect to perimeter measure, which agrees with area on the
regular part and gives $S$ measure zero; its support is $\Gamma$,
so the regular part is dense in $\Gamma$ \cite{maggi2012}.
Throughout this section, $\Gamma$, $S$, $\psi$, and $\phi$ are as
in Proposition~\ref{prop:singular-identities}, and $\Omega^*$ is the isoperimetric region.

\begin{lemma}\label{lem:singular-area}\label{lem:singular-cutoffs}
There are $r_0,C>0$ such that $|\Gamma\cap B_\rho(x)|\leq C\rho^{n-1}$
for $x\in\Gamma$ and $0<\rho<r_0$; in particular
$\sup_{q\in\Gamma}\int_\Gamma r(p,q)^{2-n}<\infty$.
\end{lemma}

\begin{proof}
There is a regular free-boundary patch: otherwise the perimeter
measure would vanish in the connected interior of $B$, forcing
$\Omega^*$ to be either empty or all of $B$ up to a null set, contrary
to $0<V<|B|$. Choose such a patch with compact closure away from $S$
and $\partial B$. A smooth deformation supported there changes
volume at a nonzero rate. The inverse function theorem therefore
allows a volume change $\delta$ to be restored at perimeter cost
at most $C|\delta|$, for $|\delta|$ small.

If $x$ lies sufficiently near $S$, a small ball $B_\rho(x)$ is
disjoint from the chosen patch and from $\partial B$. For almost
every $\rho$, deleting $\Omega^*\cap B_\rho(x)$ removes its interior
perimeter and adds at most the area of $\partial B_\rho(x)$;
this is the usual restriction formula for sets of finite perimeter
\cite{maggi2012}. The removed volume is at most $C\rho^{n}$.
Restoring it in the fixed patch and using minimality gives
$$
|\Gamma\cap B_\rho(x)|\leq C\rho^{n-1}+C\rho^{n}.
$$
Approximation of the radius gives the same bound for every small
$\rho$. Away from a neighborhood of $S$, compactness and local
$\C^{1,1}$ graphs supply a uniform bound. This proves the estimate.

On a dyadic annulus of radius $\rho$ about $q$, the integral of
$r^{2-n}$ is at most $C\rho^{2-n}\rho^{n-1}=C\rho$. Summing these
bounds proves the kernel bound near $q$, uniformly in $q$; the
remaining integral is bounded by $r_0^{2-n}|\Gamma|$.
\end{proof}

\begin{lemma}\label{lem:cutoffs}
There is a sequence $\eta_j$ of restrictions to
$\Gamma$ of smooth ambient functions such that $0\leq\eta_j\leq1$,
$\eta_j$ vanishes near $S$, $\eta_j\to1$ on $\Gamma\setminus S$,
and $\int_\Gamma|\nabla_\Gamma\eta_j|\to0$ as $j\to\infty$.
\end{lemma}

\begin{proof}
If $S\neq\varnothing$, cover it by finitely many balls
$B_{\rho_i}(x_i)$, with $x_i\in S$, $\max_i\rho_i\to0$, and
$\sum_i\rho_i^{n-2}\to0$. This is possible because
$\mathcal H^{n-2}(S)=0$ and $S$ is compact. For each ball choose a
smooth function which is zero on the ball, one outside the ball of
twice the radius, and has gradient bounded by $C/\rho_i$.
The product of these functions gives $\eta_j$, and
Lemma~\ref{lem:singular-area} implies
$$
\int_\Gamma|\nabla_\Gamma\eta_j|
\leq C\sum_i\rho_i^{-1}|\Gamma\cap B_{2\rho_i}(x_i)|
\leq C\sum_i\rho_i^{n-2}\longrightarrow0.
$$
For a fixed regular point these functions are eventually one in a
neighborhood of the point. If $S=\varnothing$, take $\eta_j=1$.
\end{proof}

The same cutoffs justify integration of divergences across $S$.
For any bounded, locally Lipschitz tangent field $Z$ on
$\Gamma\setminus S$ with $\diver_\Gamma Z\in L^1(\Gamma)$, we have
\begin{equation}\label{eq:cutoff-divergence}
\int_\Gamma\eta_j\diver_\Gamma Z
=-\int_\Gamma\langle\nabla_\Gamma\eta_j,Z\rangle
\longrightarrow0,
\end{equation}
and dominated convergence therefore gives
$\int_\Gamma\diver_\Gamma Z=0$.

\begin{lemma}\label{lem:singular-angles}
We have
$$
\sup_{q\in\Gamma\setminus S}\int_\Gamma\frac{u^2}{r^{n-1}}\,dp<\infty
\qquad\text{and}\qquad
\int_{\Gamma\times\Gamma}\frac{u^2+v^2}{r^{n-1}}<\infty.
$$
\end{lemma}

\begin{proof}
Fix $q\in\Gamma\setminus S$ and $\eps>0$, and consider the regularized
radial field
$$
Y_{q,\eps}\coloneqq
\frac{\nabla_\Gamma(r_q^2/2)}{(r_q^2+\eps^2)^{(n-1)/2}}.
$$
Since $r_q^2/2$ is smooth on $M$, this field is the tangential
projection of a smooth ambient field. Its divergence depends only
on ambient derivatives and the bounded mean curvature $H$, and is
therefore bounded for fixed $\eps$. The weak divergence formula
holds also on the $\C^{1,1}$ collar, as in the proof of
Proposition~\ref{lem:green}. Integrating against the cutoffs of
Lemma~\ref{lem:cutoffs} gives
$\int_\Gamma\diver_\Gamma Y_{q,\eps}=0$, since the cutoff error is at
most $\|Y_{q,\eps}\|_\infty\int_\Gamma|\nabla_\Gamma\eta_j|\to0$.
The identities
$\Delta_\Gamma(r_q^2/2)=n-1+r\tr_\Gamma\cE_p-H(p)ru$ and
$|\nabla_\Gamma(r_q^2/2)|^2=r^2(1-u^2)$ give
$$
\diver_\Gamma Y_{q,\eps}=
(n-1)\frac{\eps^2+r^2u^2}{(r^2+\eps^2)^{(n+1)/2}}
+\frac{r\tr_\Gamma\cE_p}{(r^2+\eps^2)^{(n-1)/2}}
-\frac{H(p)ru}{(r^2+\eps^2)^{(n-1)/2}}.
$$
Hessian comparison makes the first two terms nonnegative. Integrating
and discarding the terms not needed below, we obtain
$$
(n-1)\int_\Gamma
\frac{r^2u^2}{(r^2+\eps^2)^{(n+1)/2}}
\leq\|H\|_{L^\infty(\Gamma)}\int_\Gamma r^{2-n}\leq C,
$$
where Lemma~\ref{lem:singular-cutoffs} makes $C$ independent of $q$
and $\eps$. Monotone convergence as $\eps\downarrow0$ proves the
first assertion. Tonelli's theorem permits integration in $q$;
reversing the endpoints interchanges $u$ and $v$ and proves the
second.
\end{proof}

\begin{proof}[Proof of Proposition~\ref{prop:singular-identities}]
The argument consists of three parts. 

\emph{Part 1: absolute integrability.}
For small $r$ the functions $a,b,a',b'$ are bounded,
$|\partial_v c(r,v)|\leq C|v|$, and $|w|\leq1$; moreover
$\|\cE_p\|\leq Cr$, by the compactness of $\Gamma$ alone.
Multiplying \eqref{eq:curved-Ap-general} by $r^{n-1}$ and using
\eqref{eq:angular-potential} gives
$$
r^{n-1}\tilde{\Delta}_\Gamma\psi_q=
r\,\partial_r c(r,v)(1-u^2)+\big((n-1)u^2-rH(p)u\big)c(r,v)
+u\,\partial_v c(r,v)w+rc(r,v)\tr_\Gamma\cE_p ,
$$
and hence
\begin{equation}\label{eq:laplace-bound}
|\tilde{\Delta}_\Gamma\psi_q(p)|
\leq C\left(r^{2-n}+\frac{u^2+v^2}{r^{n-1}}\right).
\end{equation}
For $r$ bounded away from zero, \eqref{eq:curved-D} shows that
$\tilde{\Delta}_\Gamma\psi_q$ is uniformly bounded. Thus
Lemmas~\ref{lem:singular-cutoffs} and~\ref{lem:singular-angles}
give the absolute integrability of $\tilde{\Delta}_\Gamma\psi_q$ on
$\Gamma\times\Gamma$; the same holds with the endpoints reversed.
For the degree terms, $J\geq r^{n-1}$ and $\phi(0)=0$ give
$$
\frac{|v\phi(u)|}{J}
\leq \operatorname{Lip}(\phi)\frac{|uv|}{r^{n-1}}
\leq \frac{\operatorname{Lip}(\phi)}2\frac{u^2+v^2}{r^{n-1}},
$$
and the reversed term has the same bound; so both are absolutely
integrable, by Lemma~\ref{lem:singular-angles}.

\emph{Part 2: the Minkowski--Green identity.}
Fix $q\in\Gamma\setminus S$. We integrate the divergence of a
regularization of $X_q$ across $S$ by \eqref{eq:cutoff-divergence},
and then remove the regularization. For $\eps>0$ set
$$
X_{q,\eps}\coloneqq
\left(\frac{r^2}{r^2+\eps^2}\right)^{(n-1)/2}X_q
=\frac{c(r,v)\nabla_\Gamma(r_q^2/2)}
{(r^2+\eps^2)^{(n-1)/2}},
$$
extended by zero at $q$. A local $\C^{1,1}$ graph at $q$
gives $u,v=O_q(r)$ and bounded $\nabla_\Gamma v$, so $X_{q,\eps}$
is locally Lipschitz on $\Gamma\setminus S$. Its weak divergence
has no atom at $q$ and, by the product rule, equals
$$
\diver_\Gamma X_{q,\eps}=
\left(\frac{r^2}{r^2+\eps^2}\right)^{(n-1)/2}
\tilde{\Delta}_\Gamma\psi_q
+(n-1)\eps^2
\frac{c(r,v)(1-u^2)}{(r^2+\eps^2)^{(n+1)/2}},
$$
which is bounded on $\Gamma\setminus S$ for fixed $q,\eps$. Thus
\eqref{eq:cutoff-divergence} gives
\begin{equation}\label{eq:regularized-divergence}
0=\int_\Gamma
\left(\frac{r^2}{r^2+\eps^2}\right)^{(n-1)/2}
\tilde{\Delta}_\Gamma\psi_q+(n-1)\eps^2\int_\Gamma
\frac{c(r,v)(1-u^2)}{(r^2+\eps^2)^{(n+1)/2}}.
\end{equation}
We let $\eps\downarrow0$ in the two terms separately. In the first,
the integrand is dominated by $|\tilde{\Delta}_\Gamma\psi_q|$, which
is integrable on $\Gamma$: near $q$ by \eqref{eq:laplace-bound}, as
$u,v=O_q(r)$ there and $r^{2-n}$ is integrable by
Lemma~\ref{lem:singular-cutoffs}, and away from $q$ because it is
bounded. So the first term tends to
$\int_\Gamma\tilde{\Delta}_\Gamma\psi_q$ by dominated convergence.
For the second term, the area estimate in
Lemma~\ref{lem:singular-cutoffs} gives
$$
\sup_{\substack{q\in\Gamma\\0<\eps<r_0}}
\eps^2\int_\Gamma
\frac{1}{(r^2+\eps^2)^{(n+1)/2}}<\infty:
$$
the ball $r<\eps$ contributes at most a constant, each
annulus $2^j\eps\leq r<2^{j+1}\eps<r_0$ contributes at most
$C2^{-2j}$, and the remaining region is bounded as well. Now use
normal coordinates centered at $q$ and dilate by $\eps^{-1}$. On
bounded sets, the rescaled local $\C^{1,1}$ graph converges to
$T_q\Gamma$, its area element converges to Euclidean area, and
$c(r,v)(1-u^2)\to C_\psi$; the factor $\eps^{n-1}$ from the area
element cancels the factor $\eps^{1-n}$ from the kernel, the
annular estimates above control the rescaled tails, and the
contribution from $r\geq r_0$ is $O(\eps^2)$. Hence, using polar
coordinates on $T_q\Gamma$,
$$
\lim_{\eps\downarrow0}(n-1)\eps^2\int_\Gamma
\frac{c(r,v)(1-u^2)}{(r^2+\eps^2)^{(n+1)/2}}=(n-1)C_\psi|\Sph^{n-2}|
\int_0^\infty\frac{s^{n-2}}{(1+s^2)^{(n+1)/2}}\,ds
=C_\psi|\Sph^{n-2}|,
$$
where the last equality follows by integrating the derivative of
$s^{n-1}(1+s^2)^{-(n-1)/2}$, whose values at $0$ and $\infty$
are $0$ and $1$. Letting $\eps\downarrow0$ in
\eqref{eq:regularized-divergence} thus gives
$\int_\Gamma\tilde{\Delta}_\Gamma\psi_q=-C_\psi|\Sph^{n-2}|$ for
every $q\in\Gamma\setminus S$, which is the first identity; by
Step~1 it may be integrated in $q$.

\emph{Part 3: the degree identity.}
Fix a regular pole $p$. Radial projection from $p$ is Lipschitz
on $S$, because $\dist(p,S)>0$, and its image has spherical
$(n-1)$-measure zero. Its critical values on the regular part likewise
have measure zero, by the area formula on a countable collection
of regular coordinate patches. Discard these directions and the
tangent directions at $p$. Every remaining ray meets $\Gamma$
transversely in finitely many points: an infinite sequence would
accumulate in the compact boundary, either at a regular point,
contrary to transversality, at a singular point, or at $p$ along a
tangent direction, and the latter two possibilities were excluded.
Thus the signed crossing count of Section~\ref{sec:degree} applies,
with no intersection on $S$, and the Jacobian of the radial
projection is $v/J$; hence
$$
\int_\Gamma\phi(u)\frac vJ=|\Sph^{n-2}|I_\phi .
$$
Integrating in $p$ and reversing the endpoints, which the absolute
integrability of Step~1 permits, proves the second identity.
\end{proof}

\section{The Calibration Theorem}\label{sec:CMC}\label{sec:reduction}

By Propositions~\ref{lem:green} and~\ref{prop:boundary-integral},
the Minkowski--Green and degree identities hold in every
Cartan--Hadamard manifold, with masses which depend only on the
weights $\psi$, $\phi$ and on $|\Gamma|$. We now combine them,
seeking $\psi$, $\phi$ and a constant $\kappa$ for which the
combined integrand is nonnegative, and vanishes on the
chords of the unit sphere $\Sph^{n-1}\subset\R^n$, the extremal
case of Theorem~\ref{thm:CMC}. The reduction theorem below states
that such data yield Theorems~\ref{thm:main} and~\ref{thm:CMC}.

By scaling, it suffices to prove Theorem~\ref{thm:CMC} for
$H=n-1$, and to show that then $|\Gamma|\geq|\Sph^{n-1}|$; so
we assume $H=n-1$ in this section. Let $\psi$ be as in
\eqref{eq:angular-potential}, $\phi$ be a Lipschitz
function on $[-1,1]$, and $\kappa>0$ be a constant, and set
\begin{equation}\label{eq:curved-pointwise}
\tilde F\coloneqq r^{n-1}\Big(\kappa+\tilde{\Delta}_\Gamma\psi_q
+\tilde{\Delta}_\Gamma\psi_p\Big)
+\frac{r^{n-1}}{J}\big(v\phi(u)+u\phi(v)\big).
\end{equation}
By Propositions~\ref{lem:green} and~\ref{prop:boundary-integral},
\begin{equation}\label{eq:flat-strategy}
\int_{\Gamma\times\Gamma}\frac{\tilde F}{r^{n-1}}
=\kappa|\Gamma|\big(|\Gamma|-\ell\big),
\qquad
\ell\coloneqq\frac{2|\Sph^{n-2}|(C_\psi-I_\phi)}{\kappa},
\end{equation}
where the integrand is absolutely integrable by
\eqref{eq:Ap-bound} and the remarks following
Lemma~\ref{lem:diagonal-estimates}. If $\tilde F\geq0$, then
$|\Gamma|\geq\ell$. Applying this to the unit sphere
$\Gamma=\Sph^{n-1}\subset\R^n$, which has $H=n-1$, gives
$\ell\leq|\Sph^{n-1}|$, with equality exactly when $\tilde F$
vanishes on the chords of $\Sph^{n-1}$. So we seek $\psi$,
$\phi$, and $\kappa$ for which $\tilde F\geq0$ in every
Cartan--Hadamard $n$-manifold, and $\tilde F=0$ on the chords
of $\Sph^{n-1}$; then
\begin{equation}\label{eq:mass-general}
2|\Sph^{n-2}|\big(C_\psi-I_\phi\big)=\kappa|\Sph^{n-1}|,
\end{equation}
and the resulting bound $|\Gamma|\geq|\Sph^{n-1}|$ is sharp.

We call $(\psi,\phi,\kappa)$ a \emph{calibration} in dimension $n$
if $\psi$ satisfies \eqref{eq:angular-potential}, $\phi$ is Lipschitz
on $[-1,1]$ with $\phi(0)=0$, $\kappa>0$, the mass condition
\eqref{eq:mass-general} holds, and the associated integrand
$\tilde F$ defined by \eqref{eq:curved-pointwise} satisfies,
in every Cartan--Hadamard $n$-manifold and for all pairs of distinct
points equipped with unit vectors,
\begin{equation}\label{eq:curved-condition}
\tilde F\geq0,\qquad\text{with}\qquad
\tilde F=0\ \Longrightarrow\ u=v=\frac r2,\ J=r^{n-1}.
\end{equation}
For fixed $M$ and $(\psi,\phi,\kappa)$, the integrand $\tilde F$ is
a continuous function on the set of pairs
$(\nu_p,\nu_q)\in TM\times TM$ with $p\ne q$ and
$|\nu_p|=|\nu_q|=1$. Indeed, substituting \eqref{eq:curved-D}
and its endpoint-reversed version into \eqref{eq:curved-pointwise},
with $H=n-1$, expresses $\tilde F$ entirely in terms of these
vectors and the ambient geometry. 

\begin{theorem}\label{thm:general-reduction}
If there is a calibration $(\psi,\phi,\kappa)$ in dimension $n\geq3$, then
the conclusions of Theorems~\ref{thm:main} and~\ref{thm:CMC} hold
in dimension $n$.
\end{theorem}

Throughout this section we fix a calibration $(\psi,\phi,\kappa)$
in dimension $n\geq3$ and let $\tilde F$ be its associated integrand.
The CMC part of the theorem is proved in
Section~\ref{sec:CMC-int}, by integrating
\eqref{eq:curved-pointwise} over $\Gamma\times\Gamma$ via
\eqref{eq:flat-strategy} and analyzing the equality case, in which
$\tilde F$ vanishes on every chord, which forces the enclosed
domain $\Omega$ to be flat, and $\Gamma$ to be a sphere. The
isoperimetric part is proved in Section~\ref{sec:obstacle} for
isoperimetric regions with $\C^{1,1}$ boundary, which is the case
when $n\leq7$, and in Section~\ref{sec:singular} in general, where
the boundary may have singularities.

\subsection{The CMC inequality}\label{sec:CMC-int}
By \eqref{eq:mass-general}, $\ell=|\Sph^{n-1}|$ in
\eqref{eq:flat-strategy}, which thus reads
\begin{equation}\label{eq:CMC-integrated}
|\Gamma|-|\Sph^{n-1}|
=\frac{1}{\kappa|\Gamma|}\int_{\Gamma\times\Gamma}\frac{\tilde F}{r^{n-1}}
\geq0.
\end{equation}
Thus
$|\Gamma|\geq|\Sph^{n-1}|$, and rescaling yields \eqref{eq:CMC}
for an arbitrary positive constant $H$. Equation
\eqref{eq:CMC-integrated} represents the area deficit as
an integral of a nonnegative function of the chord data; this
representation will yield the rigidity statement below.

Now suppose that $H=n-1$ and $|\Gamma|=|\Sph^{n-1}|$; we show
that $\Omega$ is isometric to $\B^n$. Equality in
\eqref{eq:CMC-integrated} gives $\tilde F=0$ almost
everywhere on $\Gamma\times\Gamma$; since $\tilde F$ is
continuous and nonnegative off the diagonal, it vanishes
for all $p\neq q$, so by \eqref{eq:curved-condition}
\begin{equation}\label{eq:equality-conditions}
u=v=\frac{r}2,
\qquad J=r^{n-1}.
\end{equation}
The argument has three steps,
which follow closely the corresponding ones in
\cite[\S4.4]{chen-ghomi-wang}.

\subsubsection{Convexity}
Let $\ell$ be a complete geodesic with a chosen orientation, and
let $p,q$ be distinct points of $\ell\cap\Gamma$, with $q$ later
than $p$ along $\ell$. For the segment $\gamma$ of $\ell$ from $p$
to $q$, \eqref{eq:u-v} and \eqref{eq:equality-conditions} give
$$
\langle\gamma',\nu_p\rangle=-u(p,q)=-\frac{r(p,q)}2<0,
\qquad
\langle\gamma',\nu_q\rangle=v(p,q)=\frac{r(p,q)}2>0.
$$
So at every point of $\ell\cap\Gamma$ the geodesic crosses
$\Gamma$ transversally, and of any two such points it enters
$\Omega$ at the earlier and leaves at the later. Three points
$p_1<p_2<p_3$ of $\ell\cap\Gamma$ are therefore impossible: $p_2$
would be an exit according to the pair $(p_1,p_2)$ and an entrance
according to $(p_2,p_3)$. Hence $\ell$ meets $\Gamma$ at most
twice. If $\ell$ meets $\Omega$, then $\ell^{-1}(\Omega)$ is a
nonempty bounded open subset of $\R$, as $\Omega$ is open and
bounded, and the endpoints of each of its component intervals lie
in $\ell\cap\Gamma$; two components would need three such points.
So $\ell\cap\Omega$ is a single interval, and $\Omega$ is
geodesically convex.

\subsubsection{Flatness}
Take distinct $p,q\in\Gamma$, the chord $\gamma$ from $p$ to $q$,
and its Jacobi matrix $A(t)$. Each eigenvalue of $A(r)^TA(r)$ is
at least $r^2$ by \eqref{eq:B-lower-bound}, while their product is
$\det\big(A(r)^TA(r)\big)=J^2=r^{2(n-1)}$ by
\eqref{eq:equality-conditions}; so every eigenvalue equals $r^2$,
and $A(r)^TA(r)=r^2I$. For $\xi\neq0$ put $Z(t)\coloneqq A(t)\xi$,
which does not vanish for $t>0$, and there
\begin{equation*}
\frac{d^2}{dt^2}|Z|
=\frac{|Z'|^2|Z|^2-\langle Z,Z'\rangle^2}{|Z|^3}
-\frac{\langle Z,KZ\rangle}{|Z|}\geq0.
\end{equation*}
Thus $|Z|$ is a convex function on $[0,r]$ with $|Z(0)|=0$ and
$|Z(r)|=r|\xi|$, whence $|Z(t)|\leq t|\xi|$; combined with the
opposite inequality from \eqref{eq:B-lower-bound}, this gives
$|A(t)\xi|=t|\xi|$. So $|Z|$ is linear, its second derivative
vanishes, and both nonnegative terms above
vanish; in particular $KZ=0$. As $\xi$ is arbitrary and $A(t)$ is
invertible for $t>0$, $K(t)=0$: every plane tangent to $\gamma$
has zero sectional curvature. Now fix $x_0\in\Omega$ and a unit
vector $\xi\in T_{x_0}M$. The largest segment of the geodesic
through $(x_0,\xi)$ lying in $\Omega$ has its endpoints on
$\Gamma$, hence is a chord, so the planes tangent to it have zero
sectional curvature. Because $\Omega$ is geodesically convex, it
is star-shaped about $x_0$, and the standard Jacobi field argument
\cite[p.~157]{docarmo1992} shows that $\exp_{x_0}^{-1}$ maps
$\Omega$ isometrically onto a subset of $T_{x_0}M\simeq\R^n$ with
the Euclidean metric.

\subsubsection{Roundness}
The isometry of the previous step carries geodesics to geodesics,
so, $\Omega$ being geodesically convex, it identifies
$\overline\Omega$ with a compact convex domain in $\R^n$ bounded
by $\Gamma$; and it preserves mean curvature, so $\Gamma$ becomes
a closed embedded hypersurface of $\R^n$ with $H\equiv n-1$. By
Alexandrov's theorem \cite{alexandrov1962} such a hypersurface is
a sphere, here of radius $1$, and $\overline\Omega$ is a unit
ball. This proves the CMC part of
Theorem~\ref{thm:general-reduction}.

\subsection{The isoperimetric inequality for $3\leq n\leq 7$}
\label{sec:obstacle}\label{sec:obstacle-regular}

The isoperimetric-profile argument \cite{kleiner1992},
\cite[Thm.~7.1]{ghomi-spruck2022} reduces the isoperimetric
inequality to a mean-curvature estimate for isoperimetric regions
in a geodesic ball $B$. The boundaries of these regions have
constant mean curvature away from $\partial B$, so
Section~\ref{sec:CMC-int} does not apply to them directly. They
are $\C^{1,1}$ when $3\leq n\leq7$, which we assume in this
subsection, and may have a singular set of codimension at least
$7$ when $n\geq8$, which is treated in Section~\ref{sec:singular}.
For a domain $\Omega\subset B$ with boundary $\Gamma$, we call
$\Gamma\cap\partial B$ the \emph{contact set} and
$\Gamma\setminus\partial B$ the \emph{free part}. Since $\tilde F$
is a function of the chord and of the two oriented tangent hyperplanes
only, the hypothesis \eqref{eq:curved-condition} applies to the
pairs of points of such a boundary, although its mean curvature
is not $n-1$. We need the following extension of
Section~\ref{sec:CMC-int}.

\begin{proposition}\label{prop:obstacle}
Let $B\subset M$ be a geodesic ball, and $\Omega\subset B$ be
a domain whose boundary $\Gamma$ is a compact embedded
$\C^{1,1}$ hypersurface, smooth on the free part. Suppose
that, for some constant $H_0>0$, the mean curvature $H$ of $\Gamma$
satisfies $H=H_0$ on the free part and $H\leq H_0$
almost everywhere on the contact set. Then
\begin{equation}\label{eq:obstacle-cmc}
H_0^{n-1}\,|\Gamma|\geq (n-1)^{n-1}|\Sph^{n-1}| .
\end{equation}
\end{proposition}

\begin{proof}
After rescaling the metric, we
may assume that $H_0=n-1$, and follow
Section~\ref{sec:CMC-int}. Lemma~\ref{lem:diagonal-estimates} and
Propositions~\ref{lem:green} and~\ref{prop:boundary-integral}
hold for $\C^{1,1}$ hypersurfaces, and
\eqref{eq:curved-condition} applies to $\Gamma$, as noted above.
The only new feature here is that $H\neq n-1$ on the contact
set, whereas $\tilde F$ was built with $H=n-1$. Let $\tilde F$
be given by \eqref{eq:curved-pointwise}, which was obtained from
\eqref{eq:curved-Ap-general}
with $H=n-1$. Since the only dependence of
\eqref{eq:curved-Ap-general} on $H$ is through the term
$-\partial_r\psi\,Hu$, we have, almost
everywhere on $\Gamma\times\Gamma$,
\begin{multline}\label{eq:H-correction}
\kappa+\tilde{\Delta}_\Gamma\psi_q+\tilde{\Delta}_\Gamma\psi_p
+\frac{v\phi(u)+u\phi(v)}{J}\\
=\frac{\tilde F}{r^{n-1}}
+\partial_r\psi(r,v)\big(n-1-H(p)\big)u+\partial_r\psi(r,u)\big(n-1-H(q)\big)v .
\end{multline}
Integrating over
$\Gamma\times\Gamma$, using Proposition~\ref{lem:green} and the
degree identity \eqref{eq:boundary-integral} on the left, yields
\begin{multline*}
\kappa|\Gamma|\big(|\Gamma|-|\Sph^{n-1}|\big)
=\int_{\Gamma\times\Gamma}\frac{\tilde F}{r^{n-1}}\\
+\int_{\Gamma\times\Gamma}
\Big(\partial_r\psi(r,v)\big(n-1-H(p)\big)u+\partial_r\psi(r,u)\big(n-1-H(q)\big)v\Big),
\end{multline*}
where the integrals are absolutely convergent by
Lemma~\ref{lem:diagonal-estimates}. The first term on the right is
nonnegative by \eqref{eq:curved-condition}. We claim that the second is
nonnegative as well. Note that $n-1-H\geq0$ almost everywhere on
$\Gamma$, and $n-1-H=0$ on the free part, while $\partial_r\psi=r^{2-n}c>0$ since $c\geq a-b>0$. Furthermore, we always
have $v(p,q)=u(q,p)$. Thus it suffices to show
that $u(p,q)\geq0$ whenever $p\in\Gamma\cap\partial B$. Since $\Gamma$
is $\C^1$ and lies in $B$, it is tangent to $\partial B$ at $p$, so
$\nu_p$ is the outward normal of $B$ at $p$. Geodesic balls in a
Cartan--Hadamard manifold are convex. Hence the geodesic $\gamma$ from
$p$ to $q$ lies in $B$, so $\langle\gamma'(0),\nu_p\rangle\leq0$, and
$u=-\langle\gamma'(0),\nu_p\rangle\geq0$ by \eqref{eq:u-v}. Hence
$|\Gamma|\geq|\Sph^{n-1}|$, and undoing the normalization gives
\eqref{eq:obstacle-cmc}.
\end{proof}

In
Theorem~\ref{thm:main}, $\Omega$ is any bounded measurable set, its
\emph{volume} $|\Omega|$ is its $n$-dimensional Hausdorff measure,  and its \emph{perimeter} $|\Gamma|$ is the $(n-1)$-dimensional
Hausdorff measure of the reduced boundary $\partial^*\Omega$
\cite{maggi2012}, which is the area of $\Gamma$ when $\Gamma$ is a $\C^1$ hypersurface. Both sides of \eqref{eq:main} are
unchanged when $\Omega$ is modified on a set of measure zero, and the
equality statement is understood accordingly: equality holds only if
such a modification of $\Omega$ is a domain isometric to a Euclidean
ball.

\begin{proof}[Proof of the isoperimetric part of Theorem~\ref{thm:general-reduction} for $n\leq7$]
Let $B\subset M$ be a geodesic ball with center $o$, and let
$0<V<|B|$. By \cite[Lem.~7.2]{ghomi-spruck2022} there is an
isoperimetric region $\Omega^*\subset B$ of volume $V$, a set of
least perimeter among the subsets of $B$ with that volume, and,
because $\dim M\leq7$, its boundary $\Gamma\coloneqq\partial\Omega^*$
is $\C^{1,1}$, is smooth with constant mean curvature $H_0$ on the
free part $\Gamma\setminus\partial B$, and has $H\leq H_0$ almost
everywhere on the contact set $\Gamma\cap\partial B$. (The mean
curvature in \cite{ghomi-spruck2022} is the average of the
principal curvatures, ours is their sum.) We first check that
$H_0>0$. Put $f\coloneqq\dist(o,\cdot)^2/2$; Hessian comparison
gives $\nabla^2f\geq g$, so by \eqref{eq:laplacian-convention}
$$
\Delta_\Gamma f
=\tr_\Gamma\nabla^2f-H\langle\nabla f,\nu\rangle
\geq(n-1)-H\langle\nabla f,\nu\rangle
$$
almost everywhere on $\Gamma$. On the contact set $\nabla f=R\nu$,
with $R$ the radius of $B$, so $(H_0-H)\langle\nabla f,\nu\rangle$
is nonnegative there, and it vanishes on the free part.
Integrating over the closed hypersurface $\Gamma$,
$$
0=\int_\Gamma\Delta_\Gamma f
\geq(n-1)|\Gamma|-H_0\int_\Gamma\langle\nabla f,\nu\rangle
=(n-1)|\Gamma|-H_0\int_{\Omega^*}\Delta f ,
$$
and since $\Delta f\geq n$ this forces
$H_0\geq(n-1)|\Gamma|/\int_{\Omega^*}\Delta f>0$.
Proposition~\ref{prop:obstacle} therefore applies:
$$
H_0^{n-1}\,|\partial\Omega^*|\geq (n-1)^{n-1}|\Sph^{n-1}| .
$$

Write $\mathcal I_B(V)$ for the isoperimetric profile of $B$, so
that $\mathcal I_B(V)=|\partial\Omega^*|$, and $H_0(V)$ for the
mean curvature of the free part of the boundary of an
isoperimetric region of volume $V$. As shown in the proof of
\cite[Thm.~7.1]{ghomi-spruck2022} and the references there,
$\mathcal I_B$ is continuous and increasing, vanishes as
$V\downarrow0$, and satisfies $\mathcal I_B'(V)=H_0(V)$ for almost
every $V$, with our normalization of $H$. Consequently, for almost
every $V$, Proposition~\ref{prop:obstacle} gives
\begin{multline*}
\big(\mathcal I_B^{n/(n-1)}\big)'
=\tfrac{n}{n-1}\,\mathcal I_B^{1/(n-1)}\,\mathcal I_B'
=\tfrac{n}{n-1}\big(\mathcal I_B'^{\,n-1}\,\mathcal I_B\big)^{1/(n-1)}\\
\geq\tfrac{n}{n-1}\big((n-1)^{n-1}|\Sph^{n-1}|\big)^{1/(n-1)}
=n|\Sph^{n-1}|^{1/(n-1)}.
\end{multline*}
Because $\mathcal I_B^{n/(n-1)}$ is increasing, it dominates the
integral of its derivative,
$
\mathcal I_B(V)^{n/(n-1)}
\geq\int_0^V\big(\mathcal I_B^{n/(n-1)}\big)'
\geq n|\Sph^{n-1}|^{1/(n-1)}\,V,
$
which, as $|\Sph^{n-1}|=n|\B^n|$, is
$$
\mathcal I_B(V)\geq|\Sph^{n-1}|\Big(\frac{V}{|\B^n|}\Big)^{(n-1)/n},
$$
the Euclidean isoperimetric inequality within $B$. Every bounded
set lies in some geodesic ball, so \eqref{eq:main} follows.

Suppose finally that equality holds in \eqref{eq:main} for a
bounded set $\Omega$. Then $\Omega$ has least perimeter among
bounded sets of its volume, so it is an isoperimetric region in
every geodesic ball that contains it and, as $n<8$,
\cite[Lem.~7.2]{ghomi-spruck2022} yields a representative of
$\Omega$ whose boundary, again denoted $\Gamma$, is a smooth
compact embedded hypersurface. Equality in \eqref{eq:main} also
makes $\Omega$ a minimizer of the isoperimetric deficit
$|\partial(\cdot)|-|\Sph^{n-1}|\big(|\cdot|/|\B^n|\big)^{(n-1)/n}$
among bounded sets, whose first variation gives
$H=(n-1)|\Gamma|/(n|\Omega|)$; together with
$|\Gamma|^n=n^n|\B^n|\,|\Omega|^{n-1}$ this yields
$H^{n-1}|\Gamma|=(n-1)^{n-1}|\Gamma|^n/(n^{n-1}|\Omega|^{n-1})=(n-1)^{n-1}|\Sph^{n-1}|$.
Rescaling the metric so that $H=n-1$ and applying the rigidity
argument of Section~\ref{sec:CMC-int}, we conclude that $\Omega$
is isometric to a Euclidean ball.
\end{proof}

\subsection{The isoperimetric inequality for $n\geq 8$}
\label{sec:singular}

In dimensions $n\geq8$ the isoperimetric regions in geodesic balls may have
a singular set of codimension at least $7$, across which the
identities extend by Proposition~\ref{prop:singular-identities}.
Here we show that the argument of Section~\ref{sec:obstacle-regular}
extends as well, which completes the proof of the isoperimetric
part of Theorem~\ref{thm:general-reduction}, and hence of the whole
theorem, in every dimension. We first establish the extension of
Proposition~\ref{prop:obstacle}.

\begin{proposition}\label{prop:singular-obstacle}
Let $\Gamma$ be the boundary of an isoperimetric region in a
geodesic ball, and $H_0$ be the mean curvature of its free
part. Then $H_0>0$ and
$$
H_0^{n-1}|\Gamma|\geq (n-1)^{n-1}|\Sph^{n-1}|.
$$
\end{proposition}

\begin{proof}
Let $\Omega^*$ be the isoperimetric region with boundary $\Gamma$,
and let $S$ be its singular set.
By \cite[Lem.~7.2(i)--(ii)]{ghomi-spruck2022}, the regular free
part $\Gamma\setminus(S\cup\partial B)$ is smooth with constant
mean curvature $H_0$, and
$\Gamma$ is $\C^{1,1}$ near $\partial B$ with $H\leq H_0$ almost
everywhere on the contact set. Let $o$ be the center of $B$ and set
$f\coloneqq\dist(o,\cdot)^2/2$. The tangential gradient of $f$ and
$\Delta_\Gamma f=\tr_\Gamma\nabla^2f-H\langle\nabla f,\nu\rangle$
are bounded. Thus \eqref{eq:cutoff-divergence} gives
$\int_\Gamma\Delta_\Gamma f=0$. Since $\nabla^2f\geq g$,
$H_0-H$ vanishes on the free part, and
$\langle\nabla f,\nu\rangle\geq0$ on the contact set, the argument
of Section~\ref{sec:obstacle-regular} gives
$$
0\geq (n-1)|\Gamma|-H_0\int_\Gamma\langle\nabla f,\nu\rangle
=(n-1)|\Gamma|-H_0\int_{\Omega^*}\Delta f.
$$
The last equality is the finite-perimeter Gauss--Green formula
\cite{maggi2012}. Since $\Delta f\geq n$, the volume integral is
positive and finite, so $H_0>0$.

Rescale the metric to arrange $H_0=n-1$. The pointwise hypothesis of
Theorem~\ref{thm:general-reduction} applies to the rescaled
Cartan--Hadamard manifold as well. The identity
\eqref{eq:H-correction} holds almost everywhere on
$\Gamma\times\Gamma$, with the same function $\tilde F$. The bound
\eqref{eq:laplace-bound} holds with $n-1$ in place of $H$, so
$\tilde F/r^{n-1}$ is absolutely integrable on $\Gamma\times\Gamma$;
the other terms are, by Proposition~\ref{prop:singular-identities},
and their integrals are given by it. Consequently,
\begin{multline*}
\kappa|\Gamma|\big(|\Gamma|-|\Sph^{n-1}|\big)
=\int_{\Gamma\times\Gamma}\frac{\tilde F}{r^{n-1}}\\
+\int_{\Gamma\times\Gamma}
\Big(\partial_r\psi(r,v)(n-1-H(p))u+\partial_r\psi(r,u)(n-1-H(q))v\Big).
\end{multline*}
The first term is nonnegative by hypothesis. In the second,
$n-1-H$ is nonnegative and supported on the contact set. At a
contact point the outward normal of $\Gamma$ agrees with that of
$B$, so convexity gives $u(p,q)\geq0$ if $p$ is a contact point,
and $v(p,q)\geq0$ if $q$ is a contact point. Since $\partial_r\psi>0$,
the second term is nonnegative as well. Hence
$|\Gamma|\geq|\Sph^{n-1}|$, and rescaling proves the assertion.
\end{proof}

\begin{proof}[Proof of the isoperimetric part of Theorem~\ref{thm:general-reduction}]
Let $\mathcal I_B$ be the isoperimetric profile of a geodesic
ball $B$. The profile argument of Section~\ref{sec:obstacle-regular}
applies with Proposition~\ref{prop:singular-obstacle} in place of
Proposition~\ref{prop:obstacle}, and the two properties of
$\mathcal I_B$ it uses hold without smoothness of the minimizers.
Radial contraction about the center of $B$ is $t$-Lipschitz for
$0<t\leq1$, by Jacobi comparison, so it decreases perimeter.
The volumes of the contracted images vary continuously from zero
to the original volume. Contracting a minimizer of larger volume
therefore shows that $\mathcal I_B$ is nondecreasing. Small
geodesic balls give $\mathcal I_B(0^+)=0$. At a differentiability
point of the profile, a volume-changing variation supported in a
regular free-boundary patch touches the profile from above in
both volume directions. Its perimeter derivative divided by its
volume derivative is $H_0$, so
$\mathcal I_B'(V)=H_0(V)$; compare
\cite[proof of Thm.~7.1]{ghomi-spruck2022}.

Proposition~\ref{prop:singular-obstacle} consequently gives,
almost everywhere,
$$
\left(\mathcal I_B^{n/(n-1)}\right)'
=\frac{n}{n-1}\mathcal I_B^{1/(n-1)}H_0
\geq n|\Sph^{n-1}|^{1/(n-1)}.
$$
\begingroup\emergencystretch=2em
Since $\mathcal I_B^{n/(n-1)}$ is nondecreasing, its increase is
at least the integral of its almost-everywhere derivative.
Integrating from zero as in Section~\ref{sec:obstacle-regular} yields
$\mathcal I_B(V)^{n}\geq n^{n-1}|\Sph^{n-1}|V^{n-1}$.
Every bounded set is contained in a sufficiently large geodesic
ball. Thus \eqref{eq:main} follows in dimension $n$.\par
\endgroup

Suppose now that equality occurs for a bounded set $\Omega$.
Choose a geodesic ball whose interior contains $\overline\Omega$.
As in the equality argument of Section~\ref{sec:obstacle-regular},
$\Omega$ is a perimeter minimizer of its volume and a minimizer of
the isoperimetric deficit. Take the representative for which
$\Gamma\coloneqq\partial\Omega$ is the support of its perimeter
measure, and let $S$ be its singular set. First variations supported
in regular free-boundary patches therefore give
$H=(n-1)|\Gamma|/(n|\Omega|)$ on the regular part. Equality in the
isoperimetric inequality gives
$H^{n-1}|\Gamma|=(n-1)^{n-1}|\Sph^{n-1}|$. After rescaling, $H=n-1$ and
$|\Gamma|=|\Sph^{n-1}|$.

Proposition~\ref{prop:singular-identities} shows that the integral
of the nonnegative function $\tilde F/r^{n-1}$ is zero. Continuity on
regular pairs implies $\tilde F=0$ on every such pair, and
\eqref{eq:curved-condition} yields $u=v=r/2$ and $J=r^{n-1}$ there. We show
that this excludes $S$.
Fix one $q\in\Gamma\setminus S$. For every regular $p\neq q$,
$v(p,q)=r(p,q)/2$ gives
$$
-\left\langle\exp_q^{-1}(p),\nu_q\right\rangle
=\frac12\left|\exp_q^{-1}(p)\right|^2,
\qquad
\left|\exp_q^{-1}(p)+\nu_q\right|=1.
$$
By density of $\Gamma\setminus S$ in $\Gamma$, the last equality
holds for every $p\in\Gamma$, including $p=q$. Hence
$$
\exp_q^{-1}(\Gamma)\subset
\{\xi\in T_qM:|\xi+\nu_q|=1\}.
$$
Since $\exp_q$ is a global smooth diffeomorphism, the characteristic
function of $\exp_q^{-1}(\Omega)$ has distributional derivative
supported on this Euclidean sphere. It is therefore constant almost
everywhere on the enclosed ball and on its connected exterior.
Boundedness of $\Omega$ forces the exterior constant to be zero,
and positive volume forces the interior constant to be one. Thus,
up to a null set,
$$
\exp_q^{-1}(\Omega)=\{\xi\in T_qM:|\xi+\nu_q|<1\}.
$$
The boundary $\Gamma$ is consequently the image of the whole
sphere under $\exp_q$, so $S=\varnothing$. Consequently $\Gamma$ is smooth. The equalities for its chords
hold everywhere, and the convexity, flatness, and roundness
argument in Section~\ref{sec:CMC-int} applies without a restriction
on dimension. It identifies $\Omega$, up to a null set, with a
Euclidean unit ball in the normalized metric. Undoing the scaling
proves the equality assertion.
\end{proof}

\section{Proofs of Theorems~\ref{thm:main} and~\ref{thm:CMC}}
\label{sec:calibrations-section}\label{sec:beyond-seven}

By our calibration result, Theorem~\ref{thm:general-reduction}, our main results, Theorems~\ref{thm:main}
and~\ref{thm:CMC}, follow from the next proposition, whose proof
occupies the rest of this section.

\begin{proposition}\label{prop:existence}\label{prop:radial-calibrations}\label{prop:calibrations-7}\label{prop:calibrations-89}\label{prop:calibrations}\label{lem:G-num}\label{lem:G-positive}\label{lem:J-num}\label{lem:G-positive-89}
There is a calibration $(\psi,\phi,\kappa)$ in each dimension
$3\leq n\leq9$.
\end{proposition}

The proof is divided into three groups of dimensions. The data and
scalar verifications are given in Section~\ref{sec:G} for $n=3,5$
and in Appendix~\ref{app:calibrations} for $n=4,6,7,8,9$.
We combine the scalar inequalities with the endpoint comparisons
of Sections~\ref{subsec:jacobi} and~\ref{subsec:joint} to establish
\eqref{eq:curved-condition}. Comparison I suffices for $3\leq n\leq7$
(Sections~\ref{sec:G} and~\ref{sec:four-six-seven}), whereas
comparison II is used for $n=8,9$ (Section~\ref{subsec:integrand-89}).

\subsection{Dimensions 3 and 5}\label{sec:G}\label{sec:calibrations}\label{sec:flat-remainder}

We first record a decomposition valid in every dimension.
By \eqref{eq:curved-D}, with $H=n-1$, and since
$\partial_v c=-2bv$, the associated integrand satisfies
\begin{equation}\label{eq:curved-exact-4}
\tilde F=F+r\big(c(r,v)\tr_\Gamma\cE_p+c(r,u)\tr_\Gamma\cE_q\big)
-\Big(1-\frac{r^{n-1}}{J}\Big)\big(v\phi(u)+u\phi(v)\big),
\end{equation}
where
$$
F\coloneqq L_0-4buvw+v\phi(u)+u\phi(v),
$$
\begin{gather}\label{eq:L}
L_0(r,u,v)\coloneqq{}\kappa r^{n-1}
+r\big(\partial_r c(r,u)(1-v^2)+\partial_r c(r,v)(1-u^2)\big)\\ \notag
+(n-1)\big(v(v-r)c(r,u)+u(u-r)c(r,v)\big).
\end{gather}
In particular, $\tilde F=F$ in $\R^n$. Now we choose the calibrations
\begin{align*}
\psi(r)&=2\log r,&\phi&=0,&\kappa&=2\qquad(n=3),\\
\psi(r)&=2\log r-\frac2{r^2},&\phi(t)&=8t^3,&\kappa&=5\qquad(n=5).
\end{align*}
The mass condition \eqref{eq:mass-general} follows from
$C_\psi=2$, $I_\phi=0$ for $n=3$, and $C_\psi=4$, $I_\phi=2/3$
for $n=5$.  For $n=3$, positivity is given by
$F=2\big((r-u-v)^2+(u-v)^2\big)$.  For $n=5$, putting
$\sigma\coloneqq u+v$ and $\tau\coloneqq u-v$ gives
$$
F=(r-\sigma)^2(5r^2+2r\sigma+\sigma^2+8)
+\tau^2(2r^2+8-\tau^2).
$$
Since $|\tau|\leq2$, in both dimensions $F\geq0$, with equality
exactly when $u=v=r/2$.
It remains to verify \eqref{eq:curved-condition}.  Assume $H=n-1$,
let $\tilde F$ be as in \eqref{eq:curved-pointwise}, and put
$\theta\coloneqq1-r^{n-1}/J\in[0,1)$.  For $n=3$,
\eqref{eq:curved-exact-4}, with $c=2$ and $\phi=0$, gives
$\tilde F\geq F\geq0$.  Equality
forces $u=v=r/2$ and both traces to vanish; then $uv>0$, and
Proposition~\ref{lem:jacobi-endpoint} together with
\eqref{eq:B-lower-bound} gives $J=r^2$.
For $n=5$, since $c=2(r^2+2)$, \eqref{eq:curved-exact-4} gives
$$
\tilde F
=F+2r(r^2+2)(\tr_\Gamma\cE_p+\tr_\Gamma\cE_q)
-8uv(u^2+v^2)\theta.
$$
If $uv\leq0$, then $\tilde F\geq F\geq0$; equality would force
$F=0$, hence $u=v=r/2$ and thus $uv>0$, a contradiction.  Hence
$\tilde F>0$ in this case.
Suppose $uv>0$.  Proposition~\ref{lem:jacobi-endpoint} gives
$
\tr_\Gamma\cE_p+\tr_\Gamma\cE_q\geq\frac{2uv}{r}\theta,
$
and therefore, with $\beta\coloneqq2(u^2+v^2)-r^2-2$,
$
\tilde F\geq F-4uv\theta\beta.
$
If $\beta\leq0$, this is at least $F$.  Equality forces $F=0$, hence
$u=v=r/2$; then $\beta=-2$, so $\theta=0$ and $J=r^4$.  The exact
formula above then forces both traces to vanish.
Finally, if $\beta>0$, then $r^2<2$, since $u^2+v^2\leq2$.  As
$\theta<1$,
$$
\tilde F>
F-4uv\beta
=(r^2+6)\tau^2
+(3r^2+10)\Big(\sigma-\frac{4r(r^2+2)}{3r^2+10}\Big)^2
+\frac{r^2(16+10r^2-r^4)}{3r^2+10}>0.
$$
Thus these choices define calibrations in dimensions $3$ and $5$.

\subsection{Dimensions 4, 6, and 7}\label{sec:four-six-seven}

For these dimensions we use the lower bound
$$
L\coloneqq L_0-4b|uv|h_0,\qquad
F_0\coloneqq L+v\phi(u)+u\phi(v).
$$
Since $b\geq0$, the estimate $|w|\leq h_0$ in
\eqref{eq:zeta-bound-flat} gives $F\geq F_0$. For the choices below,
$F_0$ is polynomial on each radial branch and angular sign sector.
Put $c_p\coloneqq c(r,u)$, $c_q\coloneqq c(r,v)$, and
\begin{equation}\label{eq:G-def}
G\coloneqq L+\frac{4uvc_pc_q}{c_p+c_q}.
\end{equation}
Take the calibrations given in Sections~\ref{app:four-calibration},
\ref{app:six-calibration}, and~\ref{app:seven-calibration}.
For these, Appendix~\ref{app:calibrations} verifies
\eqref{eq:mass-general}, $F_0\geq0$ with equality exactly when
$u=v=r/2$, and $G>0$ whenever $u,v>0$; moreover $\phi\geq0$
and $\phi=0$ on $[-1,0]$.
It remains to verify \eqref{eq:curved-condition}. Assume $H=n-1$.
By \eqref{eq:curved-exact-4} and $F\geq F_0$,
\begin{equation}\label{eq:curved-bound}
\tilde F\geq
L+r\big(c_q\tr_\Gamma\cE_p+c_p\tr_\Gamma\cE_q\big)
+\frac{r^{n-1}}J\big(v\phi(u)+u\phi(v)\big).
\end{equation}
If $v\phi(u)+u\phi(v)\leq0$, then
$
\tilde F\geq
F_0+r\big(c_q\tr_\Gamma\cE_p+c_p\tr_\Gamma\cE_q\big)\geq0.
$
Equality forces $u=v=r/2$ and both traces to vanish; since $uv>0$,
Proposition~\ref{lem:jacobi-endpoint} and \eqref{eq:B-lower-bound}
then give $J=r^{n-1}$.
If $v\phi(u)+u\phi(v)>0$, then $u,v>0$, and
Proposition~\ref{lem:jacobi-endpoint} gives
$$
r\big(c_q\tr_\Gamma\cE_p+c_p\tr_\Gamma\cE_q\big)
\geq\frac{4uvc_pc_q}{c_p+c_q}
       \Big(1-\frac{r^{n-1}}J\Big),\quad
\tilde F
\geq\frac{r^{n-1}}J F_0+
\Big(1-\frac{r^{n-1}}J\Big)G\geq0.
$$
Since $G>0$, equality forces $J=r^{n-1}$ and $F_0=0$;
hence $u=v=r/2$, and \eqref{eq:curved-bound} then forces both
traces to vanish. Thus these choices define calibrations in dimensions
$4$, $6$, and $7$.

\subsection{Dimensions 8 and 9}\label{subsec:integrand-89}

Here we retain the sharper bound $|w|\leq h$. Since $b\geq0$ and
$h\leq h_0$, we have
\begin{equation}\label{eq:F-zero}
F\geq F_1\geq F_0,\qquad
F_1\coloneqq L_0-4b|uv|h+v\phi(u)+u\phi(v).
\end{equation}
Thus $F_1$ depends only on $r,u,v$. Both inequalities are equalities
when $b=0$ and on unit-sphere chords, where $w=h=h_0=1-t^2$.
A sufficient condition for the pointwise requirement in $\R^n$ is
\begin{equation}\label{eq:scalar-conditions-F}
F_1\geq0,\qquad
F_1=0\ \Longrightarrow\ u=v=\frac r2,
\end{equation}
for $r>0$ and $-1\leq u,v\leq1$.
For $u,v>0$ let
\begin{equation}\label{eq:J-def}
\mathcal J_1\coloneqq L+4buvh_0+\chi uv,\qquad
\mathcal J_2\coloneqq L+\chi\Big(1-\frac{(u-v)^2}{2}\Big),
\end{equation}
where $\chi\coloneqq2\min\{c_p,c_q\}$ for $n=8$ and
$\chi\coloneqq4c_pc_q/(c_p+c_q)$ for $n=9$; in both cases
$\chi\leq2\sqrt{c_pc_q}$.
Take $a,b,\phi,\kappa$ as in
Sections~\ref{app:eight-calibration} and~\ref{app:nine-calibration}.
As verified there, \eqref{eq:mass-general} and
\eqref{eq:scalar-conditions-F} hold, $a\geq2b$, $\phi\geq0$ with
$\phi=0$ on $[-1,0]$, and $\mathcal J_1,\mathcal J_2>0$ for
$r>0$ and $u,v>0$. It
remains to
verify \eqref{eq:curved-condition}. Assume $H=n-1$ and let
$\tilde F$ be as in \eqref{eq:curved-pointwise}. By
\eqref{eq:curved-exact-4}, with $L_0$ as in \eqref{eq:L},
\begin{equation}\label{eq:curved-exact}
\tilde F=L_0-4buvw+r\big(c(r,v)\tr_\Gamma\cE_p+c(r,u)\tr_\Gamma\cE_q\big)
+\frac{r^{n-1}}{J}\big(v\phi(u)+u\phi(v)\big),
\end{equation}
where now $w=\langle \tilde A\tilde\nu_p,\tilde\nu_q\rangle$, with
$|w|\leq h$ by \eqref{eq:zeta-bound-flat}. We show that
$\tilde F\geq0$, with equality only if $u=v=r/2$, $J=r^{n-1}$,
and $\tr_\Gamma\cE_p=\tr_\Gamma\cE_q=0$.
If $u,v$ are not both positive, then $v\phi(u)+u\phi(v)\leq0$, and \eqref{eq:curved-exact} with $|w|\leq h$, $0<r^{n-1}/J\leq1$, and $\cE_p,\cE_q\geq0$ gives $\tilde F\geq F_1\geq0$, with strict inequality since $u=v=r/2$ is impossible here. Suppose $u,v>0$. Regard $\tilde\nu_p,\tilde\nu_q$ as vectors in $\R^{n-1}$ through the parallel frame, put $w_0\coloneqq\langle\tilde\nu_p,\tilde\nu_q\rangle$, so that $|w_0|\leq h$, and apply Proposition~\ref{thm:joint} with
$$
\CC\coloneqq\begin{pmatrix}
c_q(I-\tilde\nu_p\tilde\nu_p^T)&2buv\,\tilde\nu_p\tilde\nu_q^T\\
2buv\,\tilde\nu_q\tilde\nu_p^T&c_p(I-\tilde\nu_q\tilde\nu_q^T)
\end{pmatrix}.
$$
Since $|\tilde\nu_p|^2=1-u^2$ and $|\tilde\nu_q|^2=1-v^2$, the quadratic form of $\CC$ at $(s,s')$ is $c_q(|s|^2-\langle\tilde\nu_p,s\rangle^2)+c_p(|s'|^2-\langle\tilde\nu_q,s'\rangle^2)+4buv\langle\tilde\nu_p,s\rangle\langle\tilde\nu_q,s'\rangle$, and $\CC\geq0$ if and only if $c_pc_q\geq4b^2(1-u^2)(1-v^2)$. The difference of the two sides is increasing in $a$, and at $a=2b$ equals $b^2(2u^2+2v^2-3u^2v^2)\geq0$; so $\CC\geq0$. By \eqref{eq:M-block} and $\tr_\Gamma\cE_p=\tr\bcE_p-\langle\bcE_p\tilde\nu_p,\tilde\nu_p\rangle$,
$$
\tr(\CC\M)=r\big(c_q\tr_\Gamma\cE_p+c_p\tr_\Gamma\cE_q\big)+4buv(w_0-w).
$$
Hence \eqref{eq:curved-exact} and \eqref{eq:joint} give
\begin{multline*}
\tilde F=L_0-4buvw_0+\tr(\CC\M)+\frac{r^{n-1}}{J}\big(v\phi(u)+u\phi(v)\big)\\
\geq\frac{r^{n-1}}{J}F_1+\Big(1-\frac{r^{n-1}}{J}\Big)\big(L_0-4buvw_0+\lambda(\CC)\big),
\end{multline*}
where we used $-4buvw_0\geq-4buvh$ in the term multiplied by
$r^{n-1}/J$. To bound $L_0-4buvw_0+\lambda(\CC)$, fix a unit
vector $\xi$ and let $\xi_p,\xi_q$ be the norms of the components of $\tilde\nu_p,\tilde\nu_q$ orthogonal to $\xi$, so that $\langle U\xi,\xi\rangle=c_q(u^2+\xi_p^2)$, $\langle V\xi,\xi\rangle=c_p(v^2+\xi_q^2)$, $\langle W\xi,\xi\rangle=2buv\langle\tilde\nu_p,\xi\rangle\langle\tilde\nu_q,\xi\rangle$, and $w_0-\langle\tilde\nu_p,\xi\rangle\langle\tilde\nu_q,\xi\rangle\leq\xi_p\xi_q\leq h\leq h_0$. Since $\sqrt{(u^2+\xi_p^2)(v^2+\xi_q^2)}\geq uv+\xi_p\xi_q$, the quantity in \eqref{eq:lambda-C} minus $4buvw_0$ is at least
$$
2\sqrt{c_pc_q}\,uv+\big(2\sqrt{c_pc_q}-4buv\big)\xi_p\xi_q
\geq\min\big\{\chi uv,\ \chi(uv+h_0)-4buvh_0\big\},
$$
for any $0<\chi\leq2\sqrt{c_pc_q}$, by minimizing the affine function of $\xi_p\xi_q\in[0,h_0]$. As $uv+h_0=1-(u-v)^2/2$ and $L_0-4buvh_0=L$, we conclude, with $\chi$ as in \eqref{eq:J-def},
\begin{equation}\label{eq:joint-scalars}
\tilde F\geq\frac{r^{n-1}}{J}F_1+\Big(1-\frac{r^{n-1}}{J}\Big)\min\{\mathcal J_1,\mathcal J_2\}.
\end{equation}
Since $\mathcal J_1,\mathcal J_2>0$, $\tilde F\geq0$, and equality forces $J=r^{n-1}$ and $F_1=0$, hence $u=v=r/2$; then \eqref{eq:curved-exact} with $|w|\leq h$ gives $\tilde F\geq F_1+r(c_q\tr_\Gamma\cE_p+c_p\tr_\Gamma\cE_q)$, and both traces vanish. Thus these choices define calibrations in dimensions $8$ and
$9$, which completes the proof of Proposition~\ref{prop:existence},
and hence, by Theorem~\ref{thm:general-reduction}, of
Theorems~\ref{thm:main} and~\ref{thm:CMC}.

\appendix
\section{Explicit Calibrations}\label{app:calibrations}
Here we list the calibrations used in the proof of Proposition~\ref{prop:existence} for $n\in\{4,6,7,8,9\}$. 
For every calibration, $F$ must vanish on the chords of
$\Sph^{n-1}$. Indeed, \eqref{eq:flat-strategy} and
\eqref{eq:mass-general} give
$\int_{\Sph^{n-1}\times\Sph^{n-1}}F/r^{n-1}=0$,
while \eqref{eq:curved-condition} gives $F\geq0$.
On these chords, $r=2t$, $u=v=t$, and $w=1-t^2$.
With $\tilde c(t)\coloneqq c(2t,t)$, this vanishing condition is
equivalent to
\begin{equation}\label{eq:sphere-Q}
\phi(t)=(n-1)t\tilde c(t)-(1-t^2)\tilde c'(t)
-2^{n-2}\kappa\,t^{n-2},
\qquad0<t\leq1 .
\end{equation}
So $a,b$ determine $c$ and all the dependence of $\tilde F$ on
$\psi$, since only $\partial_r\psi=r^{2-n}c$ enters the construction.
Once $\kappa$ is chosen, \eqref{eq:sphere-Q} determines $\phi$ on
$(0,1]$. The remaining choice is its extension to $[-1,0)$.
The mass condition \eqref{eq:mass-general} then follows automatically.
Indeed, multiplying \eqref{eq:sphere-Q} by
$(1-t^2)^{(n-3)/2}$ and rearranging gives
$
((1-t^2)^{(n-1)/2}\tilde c(t))'
=-(1-t^2)^{(n-3)/2}
\big(\phi(t)+2^{n-2}\kappa t^{n-2}\big).
$
Since $\tilde c(0)=C_\psi$ and
$(1-t^2)^{(n-1)/2}\tilde c(t)\to0$ as $t\uparrow1$, integration gives
$$
C_\psi-I_\phi
=2^{n-2}\kappa\int_0^1
t^{n-2}(1-t^2)^{(n-3)/2}\,dt
=\frac{\kappa|\Sph^{n-1}|}{2|\Sph^{n-2}|},
$$
which is \eqref{eq:mass-general}. 
Thus constructing
a calibration reduces to choosing $a,b,\phi,\kappa$ satisfying
\eqref{eq:sphere-Q} and \eqref{eq:curved-condition}. The weight
$\psi$ is then obtained from \eqref{eq:angular-potential}, and the
mass condition requires no separate verification.

\subsection{Exact polynomial verification}\label{app:exact-certificates}
Some of the inequalities below reduce to polynomial inequalities on compact
boxes. A polynomial whose Bernstein coefficients on a box are all
nonnegative is nonnegative on that box; the boxes, obtained by
recursive subdivision when needed, together with their coefficient
lists, form a \emph{certificate of positivity}, which can be checked
directly \cite{powers2021,boudaoud-caruso-roy2008}. The accompanying
Mathematica notebook
\cite{chen-ghomi-wang-calibrations} reconstructs the polynomials
directly from the weights displayed below and verifies all $75$
certificates in dimensions $4,6,7,8,9$ in exact rational arithmetic.
It also records the coordinate charts, denominator clearings, and
coefficient data used in the certificates. No numerical tolerance
is involved.

\subsection{Two comparison lemmas}\label{app:comparison-lemmas}
For the calibration computations below, put $c_p=c(r,u)$ and
$c_q=c(r,v)$, and let $G(r,u,v)$ be as in \eqref{eq:G-def},
now for $u,v\geq0$. The first lemma reduces
the weighted inequality to equal angles; the second extends the
short-chord inequalities to all chord lengths.

\begin{lemma}\label{lem:equal-angle-comparison}
Let $n\geq3$ and fix $r>0$. Suppose that $a>b\geq0$ and $(n-1)b-rb'>0$,
where all coefficients are evaluated at $r$. If $\delta\in\R$ satisfies
$$
\frac{(n-2)a-ra'}2+b-\frac{rb'}2-\frac{nb-rb'}8
-\frac{(n-1)^2b^2r^2}{16((n-1)b-rb')}
-\frac{2b^2}{a-b}\geq\delta,
$$
then, for $0\leq u,v\leq1$,
$$
G(r,u,v)\geq G\left(r,\frac{u+v}{2},\frac{u+v}{2}\right)
+\delta(u-v)^2.
$$
\end{lemma}

\begin{proof}
Put $\sigma=u+v$ and $\tau=u-v$. The identity
$$
\frac{4uvc_pc_q}{c_p+c_q}
=uv\left(2a-b(u^2+v^2)
-\frac{b^2(u^2-v^2)^2}{2a-b(u^2+v^2)}\right)
$$
and expansion of $L$ give
\begin{multline*}
G(r,u,v)-G(r,\sigma/2,\sigma/2)=\tau^2\Bigg(\frac{(n-2)a-ra'}2+b-\frac{rb'}2\\
+\frac{(n-1)b-rb'}4\sigma^2
-\frac{(n-1)br}{4}\sigma-\frac{nb-rb'}8\tau^2
-\frac{b^2uv\sigma^2}{2a-b(u^2+v^2)}\Bigg).
\end{multline*}
Completing the square in $\sigma$ bounds its two terms below by
$-(n-1)^2b^2r^2/(16((n-1)b-rb'))$. Also $nb-rb'>0$,
$\tau^2\leq1$, $uv\sigma^2\leq4$, and
$2a-b(u^2+v^2)\geq2(a-b)$. These bounds and the hypothesis prove
the assertion, including $\tau=0$.
\end{proof}

\begin{lemma}\label{lem:tangent-continuation}
Let $n\geq3$, and extend $a,b\in \C^1([0,2])$ by their tangent lines
at $2$. Suppose that
\begin{gather*}
a(2)>b(2)\geq0,\qquad a'(2)\geq b'(2)\geq0,
\quad\kappa2^{n-3}=a(2)-b(2)+a'(2)-b'(2),\\ 
2a'(2)+(n-1)a(2)\geq3(n-1)b(2)+6n b'(2),\quad
\kappa(n-2)2^{n-4}>2a'(2).
\end{gather*}
Put $\eta\coloneqq(n-1)(\kappa(n-2)2^{n-4}-2a'(2))>0$.
Then, for $r\geq2$ and $|u|,|v|\leq1$,
$$
F_0(r,u,v)\geq F_0(2,u,v)+\eta(r-2)^2.
$$
For $0\leq u,v\leq1$, one also has
$$
G(r,u,v)\geq G(2,u,v)+\eta(r-2)^2.
$$
\end{lemma}

\begin{proof}
The continuation preserves $b\geq0$, $a-b>0$, and
$0\leq \partial_r c(r,u)=a'(2)-b'(2)u^2\leq a'(2)$ for $|u|\leq1$.
Since the degree term is independent of $r$, it suffices to estimate
$L$. Radial derivatives at $2$ are taken from the affine branch.
The second hypothesis gives $L_r(2,1,1)=0$. On each angular sign
sector, differentiate $L_r(2,u,v)$ and use
$$
\left|1-v^2+(n-1)v(v-2)\right|\leq3(n-1),\qquad
\left|\partial_u\left[|uv|\left(1-\frac{u^2+v^2}{2}\right)\right]\right|\leq1.
$$
Together with $a'(2)\geq b'(2)\geq0$, these give
$$
\partial_u L_r(2,u,v)
\leq-2a'(2)-(n-1)a(2)+3(n-1)b(2)+6n b'(2)\leq0.
$$
The same holds in $v$. Continuity across the coordinate axes then
implies $L_r(2,u,v)\geq L_r(2,1,1)=0$ on $[-1,1]^2$.
Since the continued weights are affine,
$$
L_{rr}=\kappa(n-1)(n-2)r^{n-3}
-2(n-1)\big(v\,\partial_r c(2,u)+u\,\partial_r c(2,v)\big)\geq2\eta
\qquad(r\geq2).
$$
Integrating twice proves the assertion for $F_0$. Finally,
$4xy/(x+y)$ is increasing in each positive argument, and
$c(r,u),c(r,v)$ are nondecreasing in $r$. Its contribution to $G$
is therefore nondecreasing for $u,v\geq0$, proving the second assertion.
\end{proof}

\subsection{Dimension 4}\label{app:four-calibration}
Take $\kappa=17/16$, 
$$
\begin{aligned}
a(r)&=1+\frac38r^2-\frac{19}{64}r^3+\frac7{15}r^4
             -\frac{45}{128}r^5+\frac{119}{1440}r^6,\\
b(r)&=\frac{92}{225}-\frac{23}{16}r+\frac{153}{50}r^2
             -\frac{69}{32}r^3+\frac{17}{36}r^4,
\end{aligned}
\qquad 0\leq r\leq2,
$$
$$
a(r)=\frac{441r+292}{720},\qquad
b(r)=\frac{139r+8}{3600},\qquad r\geq2,
$$
and
$$
\phi(t)=\frac{184}{225}t_+-\frac{23}{4}t_+^2+\frac{11047}{450}t_+^3
 -27t_+^4-\frac{1486}{75}t_+^5+48t_+^6-\frac{102}{5}t_+^7,
$$
where $t_+\coloneqq\max\{t,0\}$.
The long branches are the tangent-line continuations at $r=2$, so
$a,b\in \C^1([0,\infty))$. Direct substitution of the data above into
\eqref{eq:sphere-Q}, with $n=4$, verifies it exactly, so
\eqref{eq:mass-general} follows.
The certificates, together with
positivity of the values and slopes of the long branches, give
$$
b(r)>0,\qquad a(r)-b(r)>0\quad(r\geq0),\qquad
\frac{\phi(t)}t>0\quad(0\leq t\leq1),
$$
where the quotient at $0$ means its polynomial extension.  Hence
$c(r,t)>0$ for $|t|\leq1$, while $\phi\geq0$ and $\phi=0$ on
$[-1,0]$.
For $0<r\leq2$ and $0\leq u,v\leq1$, the certificates give
\begin{equation*}
F_0(r,u,v)\geq\frac13(r-u-v)^2+\frac1{64}(u-v)^2,
\end{equation*}
so the positive short-chord sector has precisely the spherical equality set.
They also give, for nonnegative $u,v$,
$$
F_0(r,u,-v),\ F_0(r,-u,-v)
\geq \frac1{100}\max\{r/2,u,v\}^2,
$$
so the mixed and nonpositive short-chord sectors are strictly positive.
For the weighted inequality, the notebook verifies the hypotheses of
Lemma~\ref{lem:equal-angle-comparison} with $\delta=1/2$ and gives
$
G(r,t,t)\geq\frac{41}{225}\max\{r/2,t\}^2,
$
for $0<r\leq2,\ 0\leq t\leq1$.
Consequently,
$$
G(r,u,v)\geq\frac{41}{225}\max\{r/2,(u+v)/2\}^2
+\frac12(u-v)^2>0.
$$

For $r\geq2$, the displayed endpoint data satisfy
Lemma~\ref{lem:tangent-continuation}. The stationarity relation holds
exactly, the difference between the two sides of its third hypothesis
is $179/40$, and $\eta=27/10$. Thus the short-chord inequalities
extend to all $r\geq2$, with strict positivity of $F_0$ for $r>2$ and
of $G$ on the nonnegative angle sector. This proves
\eqref{eq:scalar-conditions-F} and the positivity of $G$ for $n=4$.

\subsection{Dimension 6}\label{app:six-calibration}
Take $\kappa=5/4$,
$$
\begin{aligned}
a(r)&=1+\frac58r^2-\frac5{24}r^3+\frac{81}{76}r^4
                    -\frac56r^5+\frac{243}{1216}r^6,\\
b(r)&=\frac6{19}-\frac52r+\frac{27}{4}r^2
                    -5r^3+\frac98r^4,
\end{aligned}
\qquad 0\leq r\leq2,
$$
$$
a(r)=\frac{331}{57}r-\frac{251}{38},\qquad
b(r)=\frac r2-\frac{13}{19},\qquad r\geq2,
$$
$$
\phi(t)=\frac{12}{19}t_+-10t_+^2+\frac{2093}{38}t_+^3-60t_+^4
 -\frac{1107}{19}t_+^5+\frac{400}{3}t_+^6-\frac{1089}{19}t_+^7.
$$
The long branches are the tangent-line continuations at $r=2$, so
$a,b\in \C^1([0,\infty))$. Direct substitution, with $n=6$, verifies
\eqref{eq:sphere-Q}, so \eqref{eq:mass-general} follows.
The certificates, together with
positivity of the values and slopes of the long branches, give
\begin{equation*}
b(r)>0,\qquad a(r)-b(r)>0\quad(r\geq0),\qquad
\frac{\phi(t)}t>0\quad(0\leq t\leq1).
\end{equation*}
At $t=0$ the last quotient means its polynomial extension.  Hence
$c(r,t)>0$ for $|t|\leq1$, while $\phi\geq0$ and $\phi=0$ on
$[-1,0]$.
For $0<r\leq2$ and $0\leq u,v\leq1$, the certificates give
\begin{equation*}
F_0(r,u,v)\geq\frac53(r-u-v)^2+\frac{121}{4560}(u-v)^2,
\end{equation*}
so the positive short-chord sector has precisely the spherical equality set.
The sign-sector charts and their certificates are
recorded in the notebook; they show that the mixed and nonpositive
short-chord sectors are strictly positive.

For the weighted inequality, the notebook verifies the hypotheses of
Lemma~\ref{lem:equal-angle-comparison} with $\delta=1$ and gives
$
G(r,t,t)\geq\frac{41}{76608}\max\{r/2,t\}^2
$
for $0<r\leq2,\ 0\leq t\leq1$.
Consequently,
$$
G(r,u,v)\geq\frac{41}{76608}\max\{r/2,(u+v)/2\}^2
+(u-v)^2>0.
$$
For $r\geq2$, apply Lemma~\ref{lem:tangent-continuation}.
The endpoint data satisfy the stationarity relation exactly;
the difference between the two sides of its third hypothesis is
$529/38$, and $\eta=2390/57$. Thus $F_0$ is strictly positive for
$r>2$ in every sign sector, and $G>0$ for $r\geq2$, $u,v\geq0$.
At $r=2$ the only zero of $F_0$ is $u=v=1$, already determined by
the short-chord argument. This proves
\eqref{eq:scalar-conditions-F} and the positivity of $G$ for $n=6$.

\subsection{Dimension 7}\label{app:seven-calibration}
Take $\kappa=165/128$ and
$$
a(r)=1+\frac34r^2+\frac{43}{128}r^4,\qquad
b(r)=\frac{r^2+2}{4},\qquad
\phi(t)=t_++\frac52t_+^3+\frac52t_+^5.
$$
Here $a,b$ are single polynomials, so no long-chord continuation is needed, and
$
a-b>0.
$
Direct substitution into \eqref{eq:sphere-Q}, with $n=7$, verifies
\eqref{eq:mass-general}. We record only the main reductions below; the
explicit expansions and positivity checks are given in the accompanying
Mathematica notebook \cite{chen-ghomi-wang-calibrations}.
For $u,v\geq0$, put $\sigma=u+v$ and $\tau=u-v$. Direct expansion gives
$$
128F_0=(r-\sigma)^2U_0+\tau^2V_0,
$$
where the explicit polynomials $U_0,V_0$ recorded in the notebook satisfy
$U_0>0$ and $V_0\geq270$. Thus $F_0\geq0$, with equality precisely when
$r=\sigma$ and $\tau=0$. In the mixed and nonpositive sign sectors, the
decompositions recorded there give $F_0>0$. Hence
\eqref{eq:scalar-conditions-F} holds.
For the weighted inequality, $6b-rb'=3+r^2>0$, and the elementary bounds
$$
\frac{2b^2}{a-b}\leq1,\qquad
\frac{(r^2+2)^2}{r^2+3}\leq r^2+\frac43
$$
inserted in Lemma~\ref{lem:equal-angle-comparison} give
$$
G(r,u,v)\geq G\left(r,\frac{u+v}{2},\frac{u+v}{2}\right)
+\left(\frac{25}{16}+\frac{25}{32}r^2+\frac7{256}r^4\right)(u-v)^2.
$$
It remains to consider equal angles. For $t=0$,
$G(r,0,0)=\kappa r^6+2ra'(r)>0$. For $t>0$, put $s=r/t$; then
\begin{equation*}
128G(r,t,t)=t^2\bigl(P_0(s)+t^2P_1(s)+t^4P_2(s)\bigr),
\end{equation*}
where
\begin{gather*}
P_0(s)=384(s-2)^2,\qquad
P_1(s)=8(s-2)(43s^3-58s^2-28s+40),\\
P_2(s)=3s^2q_4(s),\qquad
q_4(s)=55s^4-172s^3+86s^2+128s-64.
\end{gather*}
The notebook verifies that
$P_0(s)+t^2P_1(s)+t^4P_2(s)>0$
for all $s\geq0$ and $0<t\leq1$: elementary estimates
handle $0\leq s\leq3/5$ and $s\geq2$, while on $[3/5,2)$
positivity reduces to $q_4>0$ and
$
4P_0P_2-P_1^2=64(s-2)^2d_6(s)>0,
$
with $d_6$ and its certificate recorded there. Thus $G>0$
for $n=7$.

\subsection{Dimension 8}\label{app:eight-calibration}
Take $\kappa=1369/1000$,
$$
\begin{aligned}
a(r)&=1+\frac78r^2-\frac{483}{1480}r^3+\frac{20667}{11600}r^4
      -\frac{5724}{4625}r^5+\frac{231}{500}r^6,\\
b(r)&=\frac{1706}{3625}-\frac{1449}{370}r+\frac{10959}{1000}r^2
      -\frac{443}{50}r^3+\frac{322}{125}r^4,
\end{aligned}
\qquad 0\leq r\leq2,
$$
$$
\begin{aligned}
a(r)&=\frac{11491073}{214600}r^2-\frac{45037947}{268250}r
      +\frac{19040983}{134125},\\
b(r)&=\frac{19623}{1000}r^2-\frac{577757}{9250}r+\frac{38598}{725},
\end{aligned}
\qquad r\geq2,
$$
which are the quadratic Taylor continuations at $2$, and
$$
\phi(t)={}\frac{3412}{3625}t_+-\frac{2898}{185}t_+^2
 +\frac{642231}{7250}t_+^3-\frac{96352}{925}t_+^4
-\frac{715863}{7250}t_+^5+\frac{53224}{185}t_+^6-\frac{18928}{125}t_+^7.
$$
Thus $a,b\in \C^2([0,\infty))$. Direct substitution into
\eqref{eq:sphere-Q}, with $n=8$, verifies
it, so \eqref{eq:mass-general} follows, and 
$I_\phi=1-1369\pi/6400$. The certificates give
$$
\begin{gathered}
b>0,\qquad a-2b>0,\qquad a-b>\tfrac1{10}\quad(0\leq r\leq2),\\
\frac{\phi(t)}t>0,\qquad \phi(t)<100\quad(0\leq t\leq1).
\end{gathered}
$$
The notebook also records positive-coefficient expansions of the long
branches, which extend $b>0$ and $a>2b$ to $[0,\infty)$. Thus
$c(r,t)>0$ for $|t|\leq1$, and $\phi\geq0$ with $\phi=0$ on $[-1,0]$.
Throughout this subsection and the next, $u,v\geq0$ unless stated
otherwise, and we write
$$
S_0\coloneqq2-u^2-v^2=2h_0,\qquad
\Delta\coloneqq u^2-v^2,\qquad
D\coloneqq(r-u-v)^2+(u-v)^2,
$$
so that $h=\tfrac12\sqrt{S_0^2-\Delta^2}$ and $|\Delta|\leq S_0$.
Since $\sqrt{1-\zeta}\leq1-\zeta/2-\zeta^2/8$ for $0\leq\zeta\leq1$
(the square of the right side exceeds $1-\zeta$ by $\zeta^3/8+\zeta^4/64$),
\begin{equation*}
h\leq h_0-\frac{\Delta^2}{4S_0}-\frac{\Delta^4}{16S_0^3}\qquad(S_0>0),
\end{equation*}
where both corrections tend to zero as $S_0\to0$. Hence, by \eqref{eq:F-zero} and \eqref{eq:joint-scalars}, with
$\eps\coloneqq10^{-5}$, clearing the denominators as detailed in the
accompanying notebook reduces the required estimates to polynomial
inequalities on compact boxes. The certificates give
\begin{equation}\label{eq:reserve-eight}
F_1\geq\eps D,\qquad
\mathcal J_1\geq\eps(r^2+u^2+v^2),\qquad
\mathcal J_2\geq\eps
\qquad(0<r\leq5,\ 0\leq u,v\leq1).
\end{equation}
The first inequality extends continuously to $u=v=1$. In particular $F_1$
vanishes only on unit-sphere chords, and $\mathcal J_1,\mathcal J_2>0$ for
$u,v>0$.

For the mixed sector write $u=x\geq0$, $v=-y\leq0$. Then $h\leq h_0$
gives $F_1(r,x,-y)\geq F_0(r,x,-y)$, and the accompanying notebook certifies
$F_0(r,x,-y)\geq\eps(r^2+x^2+y^2)$ on $0\leq r\leq5$, $x,y\in[0,1]$;
reversal of the endpoints covers the other mixed sector. For
$u=-x$, $v=-y$ the degree term vanishes, and the same tests give
$F_0(r,-x,-y)\geq\eps(r^2+x^2+y^2)$. Hence $F_1>0$ for $0<r\leq5$ off
the unit-sphere chords.

It remains to treat $r\geq5$, where the long branches apply. The long-branch coefficient expansions recorded in the notebook show
$a'(r)>0$ and $(a-b)'(r)>0$ for
$r\geq5$, hence $\partial_r c(r,t)>0$ for $|t|\leq1$, so the radial derivative
terms of $L$ are nonnegative, as are the terms $(n-1)(v^2c_p+u^2c_q)$.
Since $0<c\leq a$, the remaining terms of $L_0$ are at least
$\kappa r^7-14ra$. Moreover $4|uv|h\leq1$ and $4uvh_0\leq1$, by
$2|t|\sqrt{1-t^2}\leq1$ and $u^2+v^2\geq2uv$, while
$|v\phi(u)+u\phi(v)|\leq200$. Therefore $F_1$, $\mathcal J_1$, and
$\mathcal J_2$ are all bounded below by $\kappa r^7-14ra-b-200$.
Writing $a=A_2r^2+A_1r+A_0$ and $b=B_2r^2+B_1r+B_0$ on the long
branch, where $A_1,B_1<0$ and the other coefficients are positive,
division by $r^7$ gives, for $r\geq5$,
\begin{equation*}
\frac{\kappa r^7-14ra-b-200}{r^7}
\geq\kappa-14\Big(\frac{A_2}{5^4}+\frac{A_0}{5^6}\Big)
-\frac{B_2}{5^5}-\frac{B_0}{5^7}-\frac{200}{5^7}
>\frac3{100}.
\end{equation*}
Together with \eqref{eq:reserve-eight}, this proves
\eqref{eq:scalar-conditions-F} for $F_1$, and $\mathcal J_1,\mathcal J_2>0$ for $n=8$.

\subsection{Dimension 9}\label{app:nine-calibration}
Take
$$
a(r)=\frac{143r^6+279r^4+432r^2+432}{432},\qquad
b(r)=\frac{r^2(13r^2+16)}{24},\qquad
\kappa=\frac{1261}{864},
$$
$$
\phi(t)=\frac49t_+^3\big(21+38t_+^2-26t_+^4\big).
$$
Here $a,b$ are single polynomials. The accompanying exact-arithmetic
notebook verifies $a>2b>0$ for $r>0$; moreover $a'$ and $(a-b)'$ have
positive coefficients, so $\partial_r c(r,t)>0$ for $|t|\leq1$, $r>0$.
Direct substitution into \eqref{eq:sphere-Q}, with $n=9$, verifies it, so
\eqref{eq:mass-general} follows, and explicitly $I_\phi=629/1890$.
Also $\phi\geq0$, and $\phi(t)/t\leq44/3$ on $[0,1]$: with $q=t^2$, the
derivative of $q(21+38q-26q^2)$ is $21+76q-78q^2\geq19$, so its maximum is
attained at $q=1$.
With $S_0,\Delta,D$ as in Section~\ref{app:eight-calibration}, the
first-order bound $\sqrt{1-\zeta}\leq1-\zeta/2$ gives
\begin{equation*}
h\leq h_0-\frac{\Delta^2}{4S_0}\qquad(S_0>0).
\end{equation*}
With $\eps\coloneqq10^{-4}$, this bound and the
harmonic-mean choice in \eqref{eq:J-def}, after the denominator clearings
recorded in the accompanying notebook, reduce the required estimates to
polynomial inequalities on compact boxes. The certificates give
\begin{equation*}
F_1\geq\eps D,\qquad
\mathcal J_1\geq\eps\big[D+(u+v)^2\big],\qquad
\mathcal J_2\geq\eps\qquad(0<r\leq5,\ 0\leq u,v\leq1),
\end{equation*}
with continuous extension to $u=v=1$ in the first inequality.
For the mixed sector, the accompanying notebook certifies
$F_0(r,x,-y)\geq\eps[(r-x)^2+x^2]$ for $0\leq y\leq x\leq1$ and
$0\leq r\leq5$, which is positive for $r>0$. For $x\leq y$, the terms
of $L$ linear in $r$ add to $8r(y-x)(a+bxy)\geq0$, the radial derivative
terms are nonnegative, and with $q=xy$, $S=x^2+y^2\geq2q$, and
$\phi(x)\leq(44/3)x$, the remaining terms satisfy
$$
8aS-16bq^2-4bq\Big(1-\frac S2\Big)-y\phi(x)
\geq\Big(16(a-b)-\frac{44}3\Big)q\geq\frac43q\geq0 ,
$$
so that $F_0(r,x,-y)\geq\kappa r^8>0$. For $u=-x$, $v=-y$ the degree
term vanishes and the radial terms are nonnegative, so the same
estimate applies. Reversal covers the other mixed sector. Hence $F_1>0$
for $0<r\leq5$ off the unit-sphere chords.
For the tail estimate, note that
$$
v\phi(u)+u\phi(v)\geq-\frac{44}{3}
\qquad(-1\leq u,v\leq1).
$$
The term is nonnegative when $u,v\geq0$ and vanishes when
$u,v\leq0$; in a mixed sector its negative part is at most
$(44/3)|uv|\leq44/3$.
For $r\geq5$, the elementary bounds of
Section~\ref{app:eight-calibration}, now with $16ra$ in place of $14ra$
and $44/3$ in place of $200$, bound $F_1$, $\mathcal J_1$, and
$\mathcal J_2$ below by $\kappa r^8-16ra-b-44/3$. All coefficients of
$a,b$ are positive, so replacing each negative power of $r$ by its value
at $5$ gives
$$
\frac{\kappa r^8-16ra-b-44/3}{r^8}
>\frac3{10}\qquad(r\geq5).
$$
This proves \eqref{eq:scalar-conditions-F} for $F_1$, and $\mathcal J_1,\mathcal J_2>0$ for $n=9$.

\section*{Acknowledgement}
The AI tools Claude (Anthropic) and ChatGPT (OpenAI) were used in the
preparation of this manuscript. The authors have reviewed all AI-assisted
content and take full responsibility for the final manuscript.

\bibliography{references}

\end{document}